\documentclass[11pt,twoside,reqno]{amsart}
\usepackage{amsmath}
\usepackage{epsfig}
\usepackage{amssymb}
\usepackage{xcolor}
\usepackage{blkarray}

\usepackage{version}

\usepackage{tikz}
\usepackage{calc}

\excludeversion{percent}

\usepackage{hyperref}

\newtheorem{theorem}{Theorem}[section]
\newtheorem{lemma}[theorem]{Lemma}
\newtheorem{corollary}[theorem]{Corollary}

\newtheorem{sublemma}{}[theorem]

\theoremstyle{definition}

\theoremstyle{remark}

\numberwithin{equation}{section}

\newcommand{\ba}{\backslash}

\newcommand{\cl}{{\rm cl}}

\newcommand{\subproof}{\begin{proof}[Subproof]}

\newcommand{\PP}{{\bf P}}
\newcommand{\QQ}{{\bf Q}}

\newcommand{\lm}{\lambda}

\newcommand{\wt}{\widehat}

\newcommand{\bl}{\color{black}} %% Changed from {\color{blue}} 

\begin{document}

\title[Clonal Cores and Flexipaths in matroids]{Clonal Cores and Flexipaths in matroids}

\author[N.\ Brettell]{Nick Brettell}
\address{School of Mathematics and Statistics, Victoria University of Wellington, New Zealand}
\email{nick.brettell@vuw.ac.nz}
\author[J.\ Oxley]{James Oxley}
\address{Department of Mathematics, Louisiana State University, Baton Rouge, Louisiana, USA}
\email{oxley@math.lsu.edu}
\author[C.\ Semple]{Charles Semple}
\address{School of Mathematics and Statistics, University of Canterbury, New Zealand}
\email{charles.semple@canterbury.ac.nz}
\author[G.\ Whittle]{Geoff Whittle}
\address{School of Mathematics and Statistics, Victoria University of Wellington, New Zealand}
\email{geoff.whittle@vuw.ac.nz}

\subjclass{05B35}
\date{\today}
\keywords{partitioned matroid, $4$-connected matroid, unavoidable minor}

\begin{abstract}
A partitioned matroid $(M, \{X_1,X_2,\dots,X_n\})$ consists of a matroid $M$ and a partition $\{X_1,X_2,\dots,X_n\}$ of its ground set. As such structures arise frequently in structural matroid theory, this paper introduces a general technique for analyzing those {\em special} properties of partitioned matroids that depend solely on the values of the connectivities $\lambda(X_i)$, the local connectivities $\sqcap(\cup_{j\in J}X_j, \cup_{k\in K}X_k,)$, and the dual local connectivities $\sqcap^*(\cup_{h\in H}X_h, \cup_{g\in G}X_g)$. In particular, we consider those partitioned matroids in which each $X_i$ is an independent, coindependent set of clones of cardinality $\lambda(X_i)$. Calling such partitioned matroids clonal-core matroids, we show that special results of the above type for partitioned matroids can be verified in general by proving them just for clonal-core matroids. Aiming at  the long-term goal of finding the unavoidable minors of
$4$-connected matroids, 
we illustrate this technique 
by studying $4$-paths. These are sequences $(L,P_1,P_2,\ldots, P_n,R)$ of sets that partition the ground set of a matroid so that the union of any proper initial segment of parts  is $4$-separating. Viewing the ends $L$ and $R$ as fixed, we call such a partition a $4$-flexipath if $(L,Q_1,Q_2,\ldots, Q_n,R)$ is a $4$-path for all permutations $(Q_1,Q_2,\ldots, Q_n)$ of $(P_1,P_2,\ldots, P_n)$. A straightforward simplification enables us to focus on $(4,c)$-flexipaths for some $c$ in $\{1,2,3\}$, that is, those $4$-flexipaths for which $\lm(Q_i) = c$ and $\lm(Q_i \cup Q_j) > c$  for  all distinct $i$ and $j$. Our main result for $4$-paths is that the only non-trivial case that arises here is when $c=2$. In that case, there are essentially only two possible dual pairs of $(4,c)$-flexipaths when $n \ge 5$. 
\end{abstract}

%$\sqcap(X_{i_1} \cup X_{i_2} \cup \dots X_{i_s}, X_{j_1} \cup X_{j_2} \cup \dots X_{j_t})$, and the dual local connectivities $\sqcap^*(X_{h_1} \cup X_{h_2} \cup \dots X_{h_u}, X_{g_1} \cup X_{g_2} \cup \dots X_{g_v})$

\maketitle

\section{Introduction}
\label{intro}

We are currently involved in a project  whose goal is to extend the 
results of~\cite{doov} to find the unavoidable minors of
$4$-connected matroids. While progress on this project has been steady, it is definitely ``work in progress'' and 
it would be unwise at this stage to make claims about its outcomes. Having said that, work on the
project has motivated us to solve problems that we believe are of independent interest and this is the second
paper that has arisen in that way. The first paper was motivated by the problem of deciding which version of
4-connectivity is the most appropriate for our study of unavoidable minors.
This is an issue as there are a number of versions of 4-connectivity for matroids in the 
literature. In \cite{bjosw}, we solve a problem on tangles in matroids. We argue that the results of 
\cite{bjosw} justify the use of so-called weak 4-connectivity as the most natural notion to use in our project.

In this paper, we describe a general technique for simplifying arguments in structural matroid theory and use
that technique to understand a specific structural situation that arises in our search for unavoidable minors.
We begin by giving an overview of the technique for simplifying arguments.

Recall that if $X$ and $Y$ are disjoint subsets of a matroid $M$, then the {\em local connectivity} 
$\sqcap(X,Y)$ between $X$ and $Y$ in $M$ is given by $\sqcap(X,Y)=r(X)+r(Y)-r(X \cup Y)$. The
{\em connectivity} $\lm(X)$ of $X$ is defined by $\lm(X)=\sqcap(X,E(M)-X)$. 
Most researchers in structural matroid theory will be aware of the following trick. If a theorem
for matroids can be formulated in terms of connectivity and local connectivity between members of an
associated partition of the matroid, then, to prove the theorem, it suffices to prove it for the lowest
rank examples satisfying the hypotheses. The proof for such examples is often intuitively clear,
and uses straightforward rank arguments. That is not the proof that eventually appears in a 
research paper. One has to 
convert the intuitively clear argument to one that applies in full generality. 
The conversion process is routine, but clarity
is lost and the result is far from aesthetically pleasing. 

In Section~\ref{ccore}, we give conditions that guarantee that the technique of proving the
result in a restricted domain suffices
to prove the result in full generality. Specifically, we proceed as follows. In a matroid $M$, the elements of a subset $A$ of $E(M)$ are {\em clones} 
if every permutation of $E(M)$ that fixes each element of $E(M) - A$ is 
 an automorphism of $M$. Now let $M$ be a matroid and let $\{X_1,X_2,\ldots,X_n\}$
be a partition $\mathcal X$ of $E(M)$.
We call the pair $(M,\mathcal X)$  a {\em partitioned matroid}. If   a partitioned matroid has the
additional property that each member $X_i$ of $\mathcal X$ is a clonal set of size $\lm(X_i)$,
then we say that $(M,\mathcal X)$ is a {\em clonal-core matroid}.  In essence, Corollary~\ref{core-blimey}
says that, if a conjecture can be stated in terms of connectivity, local connectivity and duals of local connectivity
of sets that respect the partition of a partitioned matroid, then, to prove the conjecture in full generality,
it suffices to prove it for the restricted class of clonal-core matroids.

Application of this technique would enable one to make considerable simplifications of arguments that 
appear in the literature. The analysis of the different types of structures that arise when one considers crossing 
3-separations in 3-connected matroids given in \cite{osw} is a clear example. We were motivated
to make explicit the technique described above by a particular problem that arose in our work on 
unavoidable minors of 4-connected matroids. For this problem, being able to reduce to clonal cores
proved to be genuinely helpful. We now turn to a discussion of this problem and its resolution.

In essence, the strategy for finding unavoidable minors of 4-connected matroids is to use extremal techniques
to gradually refine the structure of the matroid until we are finally left with the unavoidable minors
themselves. At an intermediate stage, one arrives at a matroid with an ordered 
partition $(L,P_1,P_2,\ldots, P_n,R)$ of its ground set 
into many parts where this partition induces a nested sequence of $4$-separations. In general, permuting the 
members of $(P_1,P_2,\ldots,P_n)$ destroys this property, but  not always.
Understanding the structures that arise when the property of being a nested sequence of 4-separations 
is preserved under arbitrary such permutations  plays a role in our search for unavoidable minors. 
That is the analysis that we undertake in the latter part of this paper. We note that the analysis is not
dissimilar to those given for flowers in matroids, see, for example, \cite{ao1, ao2, chen, clark, os, osw, osw2, osw3}.

%We also felt that the problem
%is potentially of independent interest and that is the motivation for presenting the results
%here. 

Before we can describe our results, we need some definitions. Our notation and terminology will follow \cite{oxbook}. 
For a positive integer $n$, we write $[n]$ for $\{1,2,\dots,n\}$. 
Let $M$ be a matroid on a set $E$. Rewriting its  connectivity function,  $\lm_M$,  introduced above, we see that 
 $\lm_M(A)=r(A)+r(E-A)-r(M)$ for all subsets $A$ of $E(M)$. 
If $X$ and $Y$ are 
disjoint subsets of $E$, then %$\kappa(X,Y)$ %, the connectivity between $X$ and $Y$,  is defined by 
$\kappa(X,Y)=\min\{\lm_M(Z):X\subseteq Z\subseteq E-Y\}$.

In the definitions that follow, we focus on the specific cases relevant to this paper. 
A {\em path of $4$-separations in $M$} is an ordered
partition $(L,P_1,P_2,\ldots,P_n,R)$ of $E(M)$ such that 
\begin{itemize}
\item[(i)] $\kappa(L,R)=3$, and
\item[(ii)] $\lambda_M(L\cup P_1 \cup P_2 \cup \dots \cup P_i)=3$ for all $i$ in $\{0,1,\dots,n\}$.
\end{itemize}
For such a path $\PP$ of $4$-separations, the members of $\PP$ are  
{\em steps},  and $L$ and $R$ are  {\em end steps} while $P_1,P_2,\ldots,P_n$ are {\em internal steps}.

The path $\PP$ is a {\em $4$-flexipath} if $(L,Q_1,Q_2,\ldots,Q_n,R)$ is also a path of $4$-separations whenever $(Q_1,Q_2,\ldots,Q_n)$ is a permutation of 
$(P_1,P_2,\ldots,P_n)$.
For a positive integer $c$,  the $4$-flexipath
$\PP$ is a $(4,c)$-{\em flexipath} if $\lm_M(P_i)=c$ for all $i$ in $[n]$, and 
$\lm_M(P_i\cup P_j)>c$ for all distinct $i,j$ in $[n]$.  Imposing these two additional constraints on $4$-flexipaths
simplifies the analysis. Moreover, descriptions of all $4$-flexipaths follow straightfowardly from 
those for $(4,c)$-flexipaths by noting that if $(L,Q_1,Q_2,\ldots,Q_n,R)$ is a $4$-flexipath $\QQ$, then so is $(L,Q_1,Q_2,\ldots,Q_{i-1}, Q_{i+1}, Q_{i+2}, \ldots Q_{n},Q_i \cup R)$. In this transformation, we say that $Q_i$ has been {\em absorbed into the right end of $\QQ$}. 

We show in Lemma~\ref{GGG} that, when $c \ge 3$,  a $(4,c)$-flexipath has at most two internal steps.
The case when $c=1$ is also straightforward. If we add the additional constraint that $M$ is 
$3$-connected, then all internal steps are singletons and these singletons are either in the 
closure or coclosure of both $L$ and $R$. The full description of $(4,1)$-flexipaths follows routinely from
these observations and is given in Corollary~\ref{414}.

This brings us to $(4,2)$-flexipaths, the most interesting case. Recall that the local connectivity  between disjoint sets $X$ and $Y$ in a matroid $M$ is given by    
$\sqcap_M(X,Y)= \sqcap(X,Y) = r(X)+r(Y)-r(X\cup Y)$. We write 
$\sqcap^*(X,Y)$ for $\sqcap_{M^*}(X,Y)$ and call this the  {\em local coconnectivity} between $X$ and $Y$ in $M$.  Let $\QQ$  be the  
$(4,2)$-flexipath $(L,Q_1,Q_2,\ldots,Q_n,R)$. 

The flexipath $\QQ$ is {\em spike-reminiscent} if all of the following hold:
\begin{itemize}
\item[(i)] $\sqcap(L,R)=1$ and $\sqcap^*(L,R)=2$;
\item[(ii)] $\sqcap(Q_i,Q_j)=1$ and $\sqcap^*(Q_i,Q_j)=0$ for all distinct $i$ and $j$ in $[n]$; and 
\item[(iii)] $\sqcap(Q_i,L)=\sqcap(Q_i,R)=1 = \sqcap^*(Q_i,L)=\sqcap^*(Q_i,R)$
for all $i$ in $[n]$.
\end{itemize}

The flexipath $\QQ$ is {\em paddle-reminiscent} if all of the following hold:
\begin{itemize}
\item[(i)] $\sqcap(L,R)=2$ and $\sqcap^*(L,R)=1$;
\item[(ii)] $\sqcap(Q_i,Q_j)=0$ and $\sqcap^*(Q_i,Q_j)=1$  for all distinct $i$ and $j$ in $[n]$; and 
\item[(iii)] $\sqcap(Q_i,L)=\sqcap(Q_i,R)=1 = \sqcap^*(Q_i,L)=\sqcap^*(Q_i,R)$
for all $i$ in $[n]$.
\end{itemize}

Illustrations of spike-reminiscent and paddle-reminiscent flexipaths are shown in Figure~\ref{reminiscent}(i) and (ii), respectively.

% Note that $\QQ$ is paddle-reminiscent in $M$ if and only if it is spike-reminiscent in $M^*$.

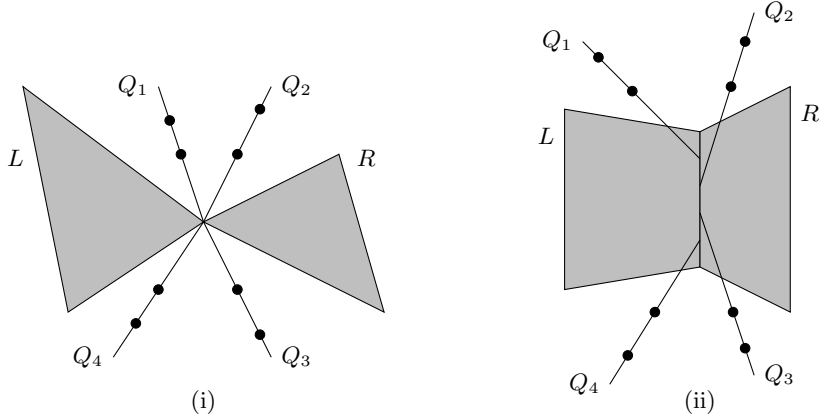
\begin{figure}
\center
\begin{tikzpicture}[scale=0.6]
%\draw[help lines] (-2.2, -1.2) grid (18.2, 10.2);

\draw[fill=lightgray] (3, 5) -- (-1, 8) -- (0, 3) -- (3, 5);
\draw[fill=lightgray] (3, 5) -- (7, 3) -- (6, 6.5) -- (3, 5);
\draw (3, 5) -- (2, 8) node[left] {\footnotesize{$Q_1$}};
\draw (3, 5) -- (4.5, 8) node[right] {\footnotesize{$Q_2$}};
\draw (3, 5) -- (4.5, 2) node[right] {\footnotesize{$Q_3$}};
\draw (3, 5) -- (1, 2) node[left] {\footnotesize{$Q_4$}};
\filldraw (2.5, 6.5) circle (3pt);
\filldraw (2.25, 7.25) circle (3pt);
\filldraw (3.75, 6.5) circle (3pt);
\filldraw (4.25, 7.5) circle (3pt);
\filldraw (3.75, 3.5) circle (3pt);
\filldraw (4.25, 2.5) circle (3pt);
\filldraw (2, 3.5) circle (3pt);
\filldraw (1.5, 2.75) circle (3pt);
\node[left, xshift=-30, yshift=-10] at (1, 7) {\footnotesize{$L$}};
\node[right, xshift=20, yshift=-10] at (5, 7) {\footnotesize{$R$}};
\node at (3, 1) {\footnotesize{(i)}};

\draw[fill=lightgray] (14, 7) -- (11, 7.5) -- (11, 3.5) -- (14, 4) -- (14, 7);
\draw[fill=lightgray] (14, 7) -- (14, 4) -- (16, 3) -- (16, 8) -- (14, 7);
\draw (14, 6.4) -- (11.4, 9) node[left] {\footnotesize{$Q_1$}};
\draw (14, 5.8) -- (15.2, 9.64) node[right] {\footnotesize{$Q_2$}};
\draw (14, 5.2) -- (15.2, 1.6) node[right] {\footnotesize{$Q_3$}};
\draw (14, 4.6) -- (12, 1.4) node[left] {\footnotesize{$Q_4$}};
\filldraw (12.5, 7.9) circle (3pt);
\filldraw (11.75, 8.65) circle (3pt);
\filldraw (14.6875, 8) circle (3pt);
\filldraw (15, 9) circle (3pt);
\filldraw (14.7333, 3) circle (3pt);
\filldraw (15, 2.2) circle (3pt);
\filldraw (13, 3) circle (3pt);
\filldraw (12.4, 2.04) circle (3pt);
\node[left, yshift=-10] at (11, 7.5) {\footnotesize{$L$}};
\node[right, yshift=-10] at (16, 8) {\footnotesize{$R$}};
\node at (14, 1) {\footnotesize{(ii)}};
\end{tikzpicture}
\caption{(i) A rank-$7$ matroid with a spike-reminiscent flexipath $(L, Q_1, Q_2, Q_3, Q_4, R)$.  (ii) A rank-$7$ matroid with a paddle-reminiscent flexipath $(L, Q_1, Q_2, Q_3, Q_4, R)$.}
\label{reminiscent}
\end{figure}

The flexipath $\QQ$ is {\em squashed} if all of the following hold:
\begin{itemize}
\item[(i)] $\sqcap(L,R)=3$ and $\sqcap^*(L,R)=0$;
\item[(ii)] $\sqcap(Q_i,Q_j)=1$ and $\sqcap^*(Q_i,Q_j)=0$ for all distinct $i$ and $j$ in $[n]$; and 
\item[(iii)] $\sqcap(Q_i,L)=\sqcap(Q_i,R)=2$, and 
$\sqcap^*(Q_i,L)=\sqcap^*(Q_i,R)=0$
for all $i$ in~$[n]$.
\end{itemize}

The flexipath $\QQ$ is {\em stretched} if all of the following hold:
\begin{itemize}
\item[(i)] $\sqcap(L,R)=0$ and $\sqcap^*(L,R)=3$;
\item[(ii)] $\sqcap(Q_i,Q_j)=0$ and $\sqcap^*(Q_i,Q_j)=1$ for all distinct $i$ and $j$ in $[n]$; and 
\item[(iii)] $\sqcap(Q_i,L)=\sqcap(Q_i,R)=0$, and 
$\sqcap^*(Q_i,L)=\sqcap^*(Q_i,R)=2$
for all $i$ in~$[n]$.
\end{itemize} 

In $\QQ$, the  step $Q_i$ is  
 {\em specially placed}  if either 
$\sqcap(L,R)=2$ and $\sqcap(L,Q_i) = 2 = \sqcap(R,Q_i)$, or %$Q_i\subseteq \cl(L) \cap  \cl(R)$
$\sqcap^*(L,R)=2$ and  $\sqcap^*(L,Q_i) = 2 = \sqcap^*(R,Q_i)$. % $Q_i\subseteq \cl^*(L) \cap \cl^*(R)$
Figure~\ref{specially} illustrates a rank-$7$ matroid in which $\{a,b\}$ is a specially placed step of the first type.  In Lemma~\ref{notmany}, we show that any $(4,2)$-flexipath has at most one specially placed step.

\begin{figure}
\center
\begin{tikzpicture}[scale=0.6]
%\draw[help lines] (-1.2, -1.2) grid (10.2, 11.2);

\draw[fill=lightgray] (4, 7.5) -- (1, 8) -- (1, 3) -- (4, 3.5) -- (4,7.5);
\draw[fill=lightgray] (4, 7.5) -- (4, 3.5) -- (6, 2.5) -- (6, 8.5) -- (4, 7.5);
\draw (4, 6.9) -- (1.4, 9.5) node[left] {\footnotesize{$Q_1$}};
\draw (4, 6.3) -- (5.2, 10.14) node[right] {\footnotesize{$Q_2$}};
\draw (4, 4.7) -- (5.2, 1.1) node[right] {\footnotesize{$Q_3$}};
\draw (4, 4.1) -- (2, 0.9) node[left] {\footnotesize{$Q_4$}};
\filldraw (2.5, 8.4) circle (3pt);
\filldraw (1.75, 9.15) circle (3pt);
\filldraw (4.6875, 8.5) circle (3pt);
\filldraw (5, 9.5) circle (3pt);
\filldraw (4.7333, 2.5) circle (3pt);
\filldraw (5, 1.7) circle (3pt);
\filldraw (3, 2.5) circle (3pt);
\filldraw (2.4, 1.54) circle (3pt);
\node[left, yshift=-10] at (1, 7.5) {\footnotesize{$L$}};
\node[right, yshift=-10] at (6, 8) {\footnotesize{$R$}};
\filldraw (4, 5.9) circle (3pt) node[left] {\footnotesize{$a$}};
\filldraw (4, 5.1) circle (3pt) node[left] {\footnotesize{$b$}};
\end{tikzpicture}
\caption{A rank-$7$ matroid in which $\{a, b\}$ is a specially placed step in the flexipath $(L, Q_1, Q_2, \{a, b\}, Q_3, Q_4, R)$.}
\label{specially}
\end{figure}
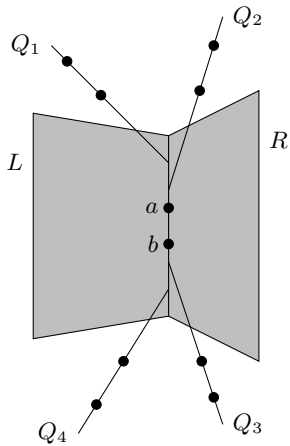

The next theorem follows from Theorem~\ref{flexi-types}, the main result of the paper. 

\begin{theorem}
\label{mainlite}
Let $\QQ$ be a $(4,2)$-flexipath with at least five internal steps. When $\QQ$ has no specially placed steps, let $\QQ'$ be $\QQ$; otherwise let $\QQ'$ be obtained from $\QQ$ by absorbing its specially placed step into its right end. Then $\QQ'$ is 
spike-reminiscent, paddle-reminiscent, squashed or stretched.
\end{theorem}

In fact,  $\QQ$ is spike-reminiscent in $M$ if and only if $\QQ$ is paddle reminiscent in $M^*$, and 
$\QQ$ is stretched in $M$ if and only if $\QQ$ is squashed in $M^*$.  It follows that, after any specially placed step is absorbed, there are at least four remaining internal steps and, up to duality, there are 
only two outcomes for $(4,2)$-flexipaths. A variety of other 
outcomes appear for $(4,2)$-flexipaths with two or three  internal steps. While these outcomes are
less interesting, it turns out to be useful to understand them for our work on unavoidable minors,
so we give a full description of them in Theorem~\ref{flexi-types}.

\section{Preliminaries}
\label{prelims}

For a matroid $M$, it is well known that  $\lm_M(X) = r(X) + r^*(X) - |X|$ for all subsets $X$ of $E(M)$. Hence $\lm_{M^*} = \lm_{M}$. When the underlying matroid is clear, we may abbreviate $\lm_M$ as $\lm$. The following basic facts about the connectivity and local connectivity functions of a matroid will be used frequently throughout the paper. The first two appear as 
Lemmas~2.6 and 2.4 of \cite{osw} and are easily verified by rewriting everything in terms of ranks of sets in $M$. The third follows straightforwardly from the first. 

\begin{lemma}
\label{geoffs}
For subsets $X$ and $Y$ of the ground set of a matroid $M$,
$$\lambda(X \cup Y) = \lambda(X) + \lambda(Y) - \sqcap(X,Y) - \sqcap^*(X,Y).$$
\end{lemma}

In the next lemma, (ii) follows from (i) by taking $D$ to be empty.

\begin{lemma}
\label{jamess}
In a matroid $M$, let $A,B,C,$ and $D$ be disjoint subsets of $E(M)$. Then 
\begin{itemize}
\item[(i)] $\sqcap(A \cup B, C \cup D) + \sqcap(A,B) + \sqcap(C,D) = \sqcap(A \cup C,B \cup D) + \sqcap(A,C) + \sqcap(B,D).$
\item[(ii)] $\sqcap(A \cup B, C) + \sqcap(A,B)  = \sqcap(A \cup C,B) + \sqcap(A,C).$
\end{itemize}
\end{lemma}

\begin{lemma}
\label{mono}
Let $(L,Q_1,Q_2,\ldots,Q_n,R)$ be a $4$-flexipath. For distinct $i$ and $j$ in $[n]$,
$$\lm(Q_i \cup Q_j) \ge \lm(Q_i ).$$
\end{lemma}

\begin{proof}
Since we have a $4$-flexipath, $\lm(L \cup Q_i \cup Q_j) = 3 = \lm(L \cup Q_i ).$ Thus, by Lemma~\ref{geoffs}, 
$\lm(L \cup Q_i \cup Q_j)  = \lm(L) + \lm(Q_i \cup Q_j) - \sqcap(L,Q_i \cup Q_j) - \sqcap^*(L,Q_i \cup Q_j) $, 
and $\lm(L \cup Q_i) = \lm(L) + \lm(Q_i) - \sqcap(L,Q_i) - \sqcap^*(L,Q_i).$ The lemma follows because 
the functions $\sqcap$ and $\sqcap^*$ are monotonic.
\end{proof}

\section{The Clonal Core of a Matroid}
\label{ccore}

The purpose of this section is to develop a versatile tool for dealing with connectivities and local connectivities of sets in a matroid. In particular, we shall define the clonal core of a matroid $M$ whose ground set has a partition $(Z_1,Z_2,\dots,Z_n)$. This clonal core $(\wt{M}, (\wt{Z}_1,\wt{Z}_2,\dots,\wt{Z}_n))$ will replace each $Z_i$ by an independent, coindependent set $\wt{Z_i}$ of clones of size $\lm(Z_i)$. We shall show that $\lm_{\wt{M}}(\wt{Z_i}) = \lm_{M}(Z_i)$ for all $i$ in $[n]$ and that, more generally, 
for all disjoint subsets $I$ and $J$ of $[n]$, we have $\sqcap_{\wt{M}}(\cup_{i \in I}\wt{Z}_i,\cup_{j \in J} \wt{Z}_j) =  \sqcap_{{M}}(\cup_{i \in I}{Z}_i,\cup_{j \in J} {Z}_j)$.

We begin with a well-known concept. For a matroid $M$, let $X$ and $Y$ be subsets of $E(M)$. We call $\{X,Y\}$ a {\it modular pair} if 
$$r_M(X) + r_M(Y) = r_M(X\cap Y) + r_M(X\cup Y).$$
A collection ${\mathcal F}$ of subsets of $E(M)$ is a {\it modular cut} of $M$ if it satisfies the following conditions.
\begin{itemize}
\item[(i)] If $X \subseteq Y \subseteq E(M)$ and $X \in {\mathcal F}$, then $Y \in {\mathcal F}$.
\item[(ii)] If $X,Y \in {\mathcal F}$ and $(X,Y)$ is a modular pair, then $X\cap Y \in {\mathcal F}$.
\item[(iii)] If $Y \in {\mathcal F}$ and $X \subseteq Y$ with $r(X) = r(Y)$, then $X \in {\mathcal F}$.
\end{itemize}

In \cite{oxbook}, a modular cut in a matroid $M$ is defined to be a set ${\mathcal F}$ of flats of $M$ obeying (i) and (ii). The definition just given extends that 
definition to arbitrary collections of subsets of $E(M)$. The next result is \cite[Lemma 6.3]{ggwT}.

\begin{lemma}
\label{ggw6.3}
Let $M$ be a matroid and $(R,S)$ be a partition of $E(M)$. Let ${\mathcal F}$ be the set of subsets $X$ of $E(M)$ for which 
$\lambda_{M/X}(R - X) = 0$. Then ${\mathcal F}$ is a modular cut of $M$.
\end{lemma}

This lemma is the basis of the  proof of the following result.

\begin{theorem}
\label{ggw6.1}
Let $M$ be a matroid and $(Z,A)$ be a partition of $E(M)$. Suppose $\lambda_M(Z) > 0$. Let $M'$ be the single-element extension of $M$ by the element $e$ corresponding to the modular cut 
$\{F \subseteq E(M): \lambda_{M/F}(Z-F) = 0\}$. Then $e$ is a non-loop element of 
$\cl_{M'}(Z) \cap \cl_{M'}(A)$ and 
$r_{M'}(X \cup \{e\}) = r_{M}(X)$ if and only if $X$ is in the modular cut.
\end{theorem}

We say that the matroid $M'$ constructed from $M$ in the last theorem has been obtained by 
{\it freely adding $e$ into the guts of $(Z,A)$}. {\bl We have symmetry between $Z$ and $A$ in the 
definition, but for our purposes it helps to focus on one side. Thus we will also say that
$M'$ has been obtained from $M$ by {\em freely adding $e$ into the guts of $Z$}. 
We may repeat the operation. 
Let $\{e_1,e_2,\dots, e_s\}$ be a set disjoint from $E(M)$. Let $M_0 = M$. 
For $i \ge 1$, inductively define $M_i$ to be the matroid that is obtained from $M_{i-1}$ by freely adding $e_i$ into the guts of 
$Z$.  Let ${\mathcal F}_{i-1}$ be the modular cut that generates $M_i$ from $M_{i-1}$.  It follows from Lemma~\ref{lem2} below that
$M_s$ is well defined in that the matroid $M_s$ does not depend on the order in which the 
elements of $\{e_1,e_2,\ldots,e_s\}$ are added. We say that $M_s$ is the matroid obtained by 
{\em freely adding $\{e_1,e_2,\ldots,e_s\}$ into the guts of $Z$.}

In the next sequence of lemmas, we shall  develop some properties of the matroids obtained
by extending freely into the guts of a partition.}  Throughout, we shall assume that $\lambda_M(Z) = t$.
%In  a matroid $M$, let $( Z,A)$ 
%be a partition of $E(M)$ such that $\lambda(Z) = t$. Let $\{e_1,e_2,\dots, e_s\}$ 
%be a set disjoint from $E(M)$. Let $M_0 = M$. 
%For $i \ge 1$, inductively define $M_i$ to be the matroid that is obtained from $M_{i-1}$ by freely adding $e_i$ into the guts of 
%{\bl $Z$}. Let ${\mathcal F}_{i-1}$ be the modular cut that generates $M_i$ from $M_{i-1}$.

\begin{lemma}
\label{lem1}
If $F \in {\mathcal F}_i$, then $F \in {\mathcal F}_j$ for all $j \ge i \ge 0$.
\end{lemma}

\begin{proof}
We argue by induction on $j-i$ noting that the result is immediate if $j - i = 0$. Assume the result holds for $j - i < n$ and let $j - i = n$. Then 
$F \in {\mathcal F}_{j-1}$. Thus $\lambda_{M_{j-1}/F}(Z-F) = 0$ and $r_{M_j} (F \cup \{e_j\}) = r_{M_{j-1}}(F)$. Hence 
$e_j$ is a loop of $M_j/F$. Thus $\lambda_{M_{j}/F}(Z-F) = 0$, so $F \in {\mathcal F}_{j}$. We conclude, by induction, that the lemma holds.
\end{proof}

\begin{lemma}
\label{lem2}
The elements $e_1,e_2,\dots, e_s$ are clones in $M_s$. 
\end{lemma}

\begin{proof}
We argue by induction on $s$ showing first that $e_1$ and $e_2$ are clones in $M_2$. Assume that this fails. Then there is a subset $S$ of $E(M)$ such that
\begin{itemize}
\item[(i)] $e_1 \in \cl_{M_2}(S)$ but $e_2 \not \in \cl_{M_2}(S)$; or 
\item[(ii)]  $e_2 \in \cl_{M_2}(S)$ but $e_1 \not \in \cl_{M_2}(S)$. 
\end{itemize}

In the first case, as $S \subseteq E(M_0)$  and $e_1 \in \cl_{M_2}(S)$, we deduce that $e_1 \in \cl_{M_1}(S)$. Thus $S \in {\mathcal F}_0$. By Lemma~\ref{lem1}, 
$S \in {\mathcal F}_1$. But this implies that $e_2 \in \cl_{M_2}(S)$, a contradiction. In case (ii), $S \in {\mathcal F}_1$, so 
$\lambda_{M_{1}/S}(Z-S) = 0$. Thus $Z - S$ is a union of components of $M_1/S$ that avoids $e_1$, so it is a union of components of $(M_1/S)\backslash e_1$, 
that is, of $M_0/S$. Thus $\lambda_{M_{0}/S}(Z-S) = 0$, so $e_1 \in \cl_{M_1}(S)$. Hence $e_1 \in \cl_{M_2}(S)$, a contradiction. We conclude that $e_1$ and $e_2$ are clones in $M_2$.

Now assume that $e_1,e_2,\dots, e_{s-1}$ are clones in $M_{s-1}$. By what we have just shown, $e_s$ and $e_{s-1}$ are clones in $M_s$. 
Say $e_s$ and $e_u$ are not clones in $M_s$ for some $u\le s-2$. Then there is a subset $V$ of $E(M_s) -\{e_u,e_s\}$ such that 
\begin{itemize}
\item[(i)] $e_u \in \cl_{M_s}(V)$ but $e_s \not \in \cl_{M_s}(V)$; or 
\item[(ii)]  $e_s \in \cl_{M_s}(V)$ but $e_u\not \in \cl_{M_s}(V)$.
\end{itemize}

In the first case,  $e_u \in \cl_{M_s}(V)$ but $e_s \not\in V$, so $e_u \in \cl_{M_{s-1}}(V)$. As $e_u$ and $e_{s-1}$ are clones in $M_{s-1}$, we deduce that $e_{s-1} \in   \cl_{M_{s-1}}(V)$. 
Hence $e_{s-1} \in   \cl_{M_{s}}(V)$. As $e_{s-1}$ and $e_s$ are clones in $M_s$, it follows that $e_s \in   \cl_{M_{s}}(V)$, a contradiction. 
In the second case, $e_{s-1} \in \cl_{M_s}(V)$. As $e_s \not \in V$, it follows that $e_{s-1} \in \cl_{M_{s-1}}(V)$. Hence $e_{u} \in \cl_{M_{s-1}}(V)$ and $e_u \in \cl_{M_{s}}(V)$, a contradiction. We conclude that 
 $e_1,e_2,\dots, e_s$ are clones in $M_s$ and the lemma follows by induction.
\end{proof}

\begin{lemma}
\label{lem3}
$\lambda_{M_s}(Z) = t$.
\end{lemma}

\begin{proof}
Since  $A \in {\mathcal F}_0$, we see that $e_1 \in \cl_{M_1}(A)$, so $e_1 \in \cl_{M_s}(A)$. As  $e_1,e_2,\dots, e_s$ are clones in $M_s$, we see that 
 $\{e_1,e_2,\dots, e_s\} \in \cl_{M_s}(A)$. Thus $\lambda_{M_s}(Z) = \lambda_M(Z) = t$. 
\end{proof}

\begin{lemma}
\label{lem4}
$r_{M_s}(\{e_1,e_2,\dots, e_s\}) \le t$.
\end{lemma}

\begin{proof}
As $Z \in {\mathcal F}_0$, we see that $e_1 \in \cl_{M_1}(Z)$, so 
$\{e_1,e_2,\dots, e_s\} \subseteq \cl_{M_s}(Z)$. By submodularity, 
\begin{eqnarray*}
r_{M_s}(\{e_1,e_2,\dots, e_s\})  & \le &r_{M_s}(A \cup \{e_1,e_2,\dots, e_s\}) \\
&  &+ r_{M_s}(Z \cup \{e_1,e_2,\dots, e_s\}) - r(M_s)\\
						& = &r_{M}(A) + r_{M}(Z) - r(M)= t.
						\end{eqnarray*}
\end{proof}

\begin{lemma}
\label{lem5} 
For all $u \le t$, the set 
$\{e_1,e_2,\dots, e_u\}$ is independent in $M_s$.
\end{lemma}

\begin{proof}
Let $X_i = \{e_1,e_2,\dots,e_i\}$. It suffices to prove that $e_{i+1} \not\in \cl_{M_{i+1}}(X_i)$ when $i+1 \le s \le t$. Assume the contrary. Then $X_i \in {\mathcal F}_i$. 
Thus 
\begin{eqnarray*}
0& = & r_{M_i/X_i}(Z) + r_{M_i/X_i}(A) - r(M_i/X_i)\\
& = & r_{M_i}(Z \cup X_i) +  r_{M_i}(A \cup X_i)  - r(M_i)- r_{M_i}(X_i)\\
& = & r_M(Z) + r_M(A) - r(M) -  r_{M_i}(X_i)\\
& = & \lambda_M(Z) - r_{M_i}(X_i).
						\end{eqnarray*}
Hence $t = \lambda_M(Z) = r_{M_i}(X_i) \le i < u \le t$, a contradiction. 
\end{proof}

\begin{lemma}
\label{lem6}  
If $X \subseteq A$ and $\cl_{M_s}(X) \cap \{e_1,e_2,\dots, e_s\} \neq \emptyset$, then $\sqcap(X,Z) = t$.
\end{lemma}

\begin{proof}
Because $e_1,e_2,\dots, e_s$ are clones in $M_s$, we may assume that $e_1 \in \cl_{M_s}(X)$. Hence 
 $e_1 \in \cl_{M_1}(X)$. As ${\mathcal F}_0$ is the modular cut that generates $M_1$ from $M$, it follows that $X \in {\mathcal F}_0$. Thus 
 \begin{eqnarray*}
0& = & r_{M/X}(A-X) + r_{M/X}(Z) - r(M/X)\\
& = & r_{M}(A) +  r_{M}(Z \cup X)  - r(M)- r_{M}(X)\\
& = & (r_M(A) + r_M(Z) - r(M)) - (r_M(Z) + r_M(X) - r_M(Z \cup X))\\
& = & \lambda_M(Z) - \sqcap_M(Z,X).
						\end{eqnarray*}
We deduce that $t = \lambda_M(Z) = \sqcap_M(Z,X)$.
\end{proof}

{\bl The case that is of most interest to us is the case when $s=t$. The next result captures some key properties in this case. We state the full set of hypotheses.

\begin{theorem}
\label{first} 
Let $M$ be a matroid and $(Z,A)$ be a partition of its ground set for which 
$\lambda_M(Z) = t  > 0$. Let  $M_t$ denote the matroid obtained by  freely adding the set
$\{e_1,e_2,\dots,e_t\}$ into the guts of $Z$. Then $\{e_1,e_2,\ldots, e_t\}$ is an independent set of 
clones in $M_t$. Moreover
$\cl_{M_t}(A) \cap \cl_{M_t}(Z)$ contains and is spanned by $\{e_1,e_2,\dots,e_t\}$.
\end{theorem}
}

\begin{proof}
By Lemmas~\ref{lem2}, \ref{lem4}, and \ref{lem5}, $\{e_1,e_2,\dots, e_t\}$ is a rank-$t$ set of clones in $M_t$. 
As $e_i \in \cl_{M_i}(Z) \cap \cl_{M_i}(A \cup \{e_1,e_2,\dots, e_{i-1}\})$ for all $i$, we see that $e_i \in \cl_{M_i}(Z) \cap \cl_{M_i}(A)$. Hence 
$\cl_{M_t}(Z) = \cl_M(Z) \cup  \{e_1,e_2,\dots,e_t\}$  and $\cl_{M_t}(A) = \cl_M(A) \cup  \{e_1,e_2,\dots,e_t\}$. 
Thus $r(\cl_{M_t}(Z)) = r_M(Z)$ and $r(\cl_{M_t}(A)) = r_M(A)$. Since $\{e_1,e_2,\dots,e_t\} \subseteq \cl_{M_t}(Z) \cap \cl_{M_t}(A)$, we deduce that 
 \begin{eqnarray*}
t& \le & r_{M_t}(\cl_{M_t}(A) \cap \cl_{M_t}(Z))\\
& = & r_{M}((\cl_{M}(A) \cap \cl_{M}(Z)) \cup \{e_1,e_2,\dots,e_t\}) \\
& = & r_{M}((\cl_{M}(A) \cap \cl_{M}(Z))\\
& \le  & r_{M}((\cl_{M}(A)) +  r_M(\cl_{M}(Z)) -r(M)\\
& = & r_M(A) + r_M(Z) - r(M)\\
& = & \lambda_M(Z)  = t.
\end{eqnarray*}
Hence 
$\{e_1,e_2,\dots,e_t\}$ is, indeed,  a basis for $M_t|(\cl_{M_t}(A) \cap \cl_{M_t}(Z)).$
\end{proof}

{\bl For the next five results, we remain under the hypotheses of Theorem~\ref{first}.
Let $\{e_1,e_2,\dots,e_t\} = G$. Then the ground set of the matroid $M_t$  constructed in the last theorem  is the disjoint union of $Z$, $A$, and $G$.}

\begin{lemma}
\label{G1} 
Let $X \subseteq A$. Then $r_{M/Z}(X) = r_{M/G}(X)$, that is, 
$$r_{M_t}(X \cup Z) - r_{M_t}(Z) = r_{M_t}(X \cup G) - r_{M_t}(G).$$
\end{lemma}

\begin{proof}
By Theorem~\ref{first}, 
$G$ spans $\cl_{M_t}(Z) \cap \cl_{M_t}(A)$. Since $\lambda_M(Z) = t = r_{M_t}(G)$, we deduce that $\lambda_{M_t/G}(Z) = 0$. 
Thus 
$$M_t/G\ba Z = M_t/G/Z = (M_t/Z)/G = (M_t/Z)\ba G$$ where the last step holds because $G \subseteq \cl_{M_t}(Z)$. 
Hence $r_{M_t/G\ba Z}(X) = r_{M_t/Z\ba G}(X)$, that is, $r_{M/Z}(X) = r_{M/G}(X)$.
\end{proof}

\begin{lemma}
\label{G2} 
Let $X$ and $Y$ be disjoint subsets of $A$. Then 
\begin{itemize}
\item[(i)] $\sqcap_M(X,Y) = \sqcap_{M_t \ba Z}(X,Y)$; and 
\item[(ii)] $\sqcap_M(X \cup Z,Y) = \sqcap_{M_t \ba Z}(X \cup G,Y)$.
\end{itemize}
\end{lemma}

\begin{proof}
Part (i) is immediate since $M$ is a restriction of $M_t$. For (ii), we have 
\begin{eqnarray*}
\sqcap_M(X \cup Z,Y)& = & r_{M}(X \cup Z) + r_M(Y) - r_M(X \cup Z \cup Y)\\
& = & (r_{M}(X \cup Z) - r_M(Z)) + r_M(Y)\\ 
&& ~~~~~ - (r_M(X \cup Z \cup Y) - r_M(Z))\\
& = & (r_{M_t}(X \cup G) - r_{M_t}(G)) + r_{M_t}(Y)\\
&& ~~~~~ - (r_{M_t}(X \cup G \cup Y) - r_{M_t}(G))\\
& =  & r_{M_t}(X \cup G) + r_{M_t}(Y) - r_{M_t}(X \cup G \cup Y)\\
& = &  \sqcap_{M_t}(X \cup G,Y)\\
& = & \sqcap_{M_t \ba Z}(X \cup G,Y)
\end{eqnarray*}
where the third step follows from two applications of Lemma~\ref{G1}.
\end{proof}

\begin{corollary}
\label{G2cor} 
Suppose $F \subseteq A$. Then $F \cup Z$ is a flat of $M$ if and only if $F \cup G$ is a flat of $M_t \ba Z$.
\end{corollary}

\begin{proof}
Take $e$ in $A - F$. By Lemma~\ref{G2}(ii), 
$\sqcap_M(F \cup Z,\{e\}) = \sqcap_{M_t \ba Z}(F\cup G,\{e\}).$ 
Thus $e \in \cl_M(F \cup Z)$ if and only if $e \in \cl_{M_t \ba Z}(F \cup G)$. The result follows.
\end{proof}

\begin{lemma}
\label{G2*}
Let $X$ and $Y$ be disjoint subsets of $A$. Then 
\begin{itemize}
\item[(i)] $\sqcap^*_M(X,Y) = \sqcap^*_{M_t \ba Z}(X,Y)$; and 
\item[(ii)] $\sqcap^*_M(X \cup Z,Y) = \sqcap^*_{M_t \ba Z}(X \cup G,Y)$.
\end{itemize}
\end{lemma}

\begin{proof}
Suppose $X' \in \{X, X \cup G\}$. Then 
\begin{eqnarray}
\label{dagger}
\sqcap^*_{M_t\ba Z}(X',Y)& = & \sqcap_{(M_t\ba Z)^*}(X',Y) \nonumber \\
 &= &\sqcap_{M^*_t/Z}(X',Y)\nonumber  \\ 
 &= &r_{M_t^*/Z}(X')+  r_{M_t^*/Z}(Y) - r_{M_t^*/Z}(X' \cup Y) \nonumber \\
& = & r_{M_t^*}(X'\cup Z)+  r_{M_t^*}(Y\cup Z) \nonumber \\ 
&& ~~~~ - r_{M_t^*}(X' \cup Y \cup Z) - r_{M_t^*}(Z) \nonumber \\
& = & r(M_t \ba (X'\cup Z))+  r(M_t \ba(Y\cup Z)) \nonumber \\ 
&&   - r(M_t\ba (X' \cup Y \cup Z)) -  r(M_t \ba Z).
\end{eqnarray}
Thus, recalling that $E(M_t)$ is the disjoint union of $Z, A$, and $G$, we have  
\begin{eqnarray}
\label{star}
\sqcap^*_{M_t\ba Z}(X,Y)& = & r_{M_t}((A-X) \cup G)+ r_{M_t}((A-Y) \cup G) \nonumber \\ 
&&   -  r_{M_t}((A-(X\cup Y)) \cup G) -  r_{M_t}(A \cup G) \nonumber \\ 
&= & r_{M_t}((A-X) \cup Z)+ r_{M_t}((A-Y) \cup Z) \nonumber \\ 
&&   -  r_{M_t}((A-(X\cup Y)) \cup Z) -  r_{M_t}(A \cup Z) 
\end{eqnarray}
where the last step follows by four applications of Lemma~\ref{G1}.

For $X''$ in $\{X,X\cup Z\}$, we have 
\begin{eqnarray}
\label{2stars}
\sqcap^*_{M}(X'',Y)& = &  r^*_{M}(X'')+   r^*_{M}(Y) - r^*_{M}(X'' \cup Y) \nonumber \\
& = & |X''| + r(E(M) - X'') +  |Y| + r(E(M) - Y) - |X'' \cup Y|  \nonumber \\
 && - r(E(M) - (X''\cup Y)) - r(M).
 \end{eqnarray}
 Thus 
 $$\sqcap^*_{M}(X,Y) = r((A-X) \cup Z) + r((A-Y) \cup Z) - r((A-(X\cup Y)) \cup Z) - r(M).$$
 Therefore, by (\ref{star}), $$\sqcap^*_M(X,Y) = \sqcap^*_{M_t\ba Z}(X,Y),$$ that is, (i) holds.
 
 Now, by (\ref{2stars}), 
 $$\sqcap^*_{M}(X\cup Z,Y) = r(A-X) +r((A-Y) \cup Z) - r(A - (X \cup Y)) - r(A \cup Z).$$ 
 Moreover, by (\ref{dagger}),
 \begin{eqnarray*}
 \sqcap^*_{M_t\ba Z}(X \cup G,Y)& = &  r_{M_t}(A-X)+ r_{M_t}((A-Y) \cup G) \\ 
&&   -  r_{M_t}(A-(X\cup Y)) -  r_{M_t}(A \cup G)\\ 
& = &  r_{M}(A-X)+ r_{M}((A-Y) \cup Z) \\ 
&&   -  r_{M}(A-(X\cup Y)) -  r_{M}(A \cup Z)
\end{eqnarray*} 
by two applications of Lemma~\ref{G1}. We conclude that  
$\sqcap^*_M(X \cup Z,Y) = \sqcap^*_{M_t \ba Z}(X \cup G,Y)$, that is (ii) holds.
\end{proof}

\begin{lemma}
\label{G3}
In $M_t \ba Z$, the set $G$ is an independent, coindependent set of clones of cardinality $\lambda_M(Z)$, and 
$\lambda_{M_t\ba Z} (G) = \lambda_M(Z)$.
\end{lemma}

\begin{proof}
By Theorem~\ref{first}, $G$ is an independent set of clones in $M_t$ of cardinality $\lambda_M(Z)$. Now 
 \begin{eqnarray*}
r^*_{M_t\ba Z}(G)& = &  r_{M_t^*/Z}(G) \\ 
& = &   r_{M_t^*}(G \cup Z) - r_{M_t^*}(Z)\\
& = & |G \cup Z| + r_{M_t}(A) - r(M_t)\\
&&  - (|Z| + r_{M_t}(A \cup G) -r(M_t))\\
& = & |G|
\end{eqnarray*} 
where the last step holds because $A$ spans $G$ in $M_t$. Thus $G$ is coindependent in $M_t \ba Z$. Finally, 
\begin{eqnarray*}
\lambda_{M_t\ba Z}(G)& = &  r_{M_t\ba Z}(G) + r^*_{M_t\ba Z}(G) - |G|\\
& = & |G|\\
& = & \lambda_M(Z)
\end{eqnarray*} 
where the second step follows because $G$ is independent and coindependent in $M_t \ba Z$.
\end{proof}

{\bl We now consider freely adding different elements into the guts of disjoint sets in a matroid. 

\begin{lemma}
\label{contract-ok}
Let $M$ be a matroid and suppose $A\subseteq E(M)$. Let $M_{\langle a \rangle}$ be the matroid that is obtained from $M$ 
by freely adding $a$ into the guts of $A$. If $Y\subseteq E(M)$, then $M_{\langle a \rangle}/Y$
is obtained from $M/Y$ by freely adding $a$ into the guts of $A-Y$.
\end{lemma}

\begin{proof}
Say $X\subseteq E(M)-Y$. Then $a\in \cl_{M_{\langle a \rangle}/Y}(X)$ if and only if $a\in\cl_{M_{\langle a \rangle}}(X\cup Y)$. From the definition
of $M_{\langle a \rangle}$, Theorem~\ref{ggw6.1} implies that the latter holds if and only if $\lm_{M/(X\cup Y)}(A-(X\cup Y))=0$, that is,
if and only if $\lm_{(M/Y)/X}((A-Y)-X)=0$.  But $\lm_{(M/Y)/X}((A-Y)-X)=0$ if and only if 
$X$ is in the modular cut that corresponds to freely adding $a$ 
into the guts of $A-Y$ in $M/Y$. 
\end{proof}

\begin{lemma}
\label{different-guts}
Let $M$ be a matroid, let $A$ and $B$ be disjoint subsets of $E(M)$ and let $\{a,b\}$ be disjoint from $E(M)$.
Let $M_{\langle a \rangle}$ be the matroid that is obtained from $M$ by freely adding $a$ into the guts of $A$;
let $M_{\langle b \rangle}$ be obtained from $M$ by freely adding $b$ into the guts of $B$; let $M_{{\langle a \rangle}{\langle b \rangle}}$ be the matroid that is obtained 
from $M_{\langle a \rangle}$ by freely adding $b$ into the guts of $B$; and let $M_{{\langle b \rangle}{\langle a \rangle}}$ be the matroid that is obtained from $M_{\langle b \rangle}$ by freely adding 
$a$ into the guts of $A$. Then $M_{{\langle a \rangle}{\langle b \rangle}}=M_{{\langle b \rangle}{\langle a \rangle}}$. 
\end{lemma}

\begin{proof} 
Clearly $M_{{\langle a \rangle}{\langle b \rangle}}\ba b=M_{\langle a \rangle}$ and $M_{{\langle b \rangle}{\langle a \rangle}}\ba a=M_{\langle b \rangle}$.
Next we show that 

\begin{sublemma}
\label{subdifferent-1}
$M_{{\langle a \rangle}{\langle b \rangle}}\ba a=M_{\langle b \rangle}$ and $M_{{\langle b \rangle}{\langle a \rangle}}\ba b=M_{\langle a \rangle}$.
\end{sublemma}

Suppose $X\subseteq E(M)$. We prove that $b\in\cl_{M_{{\langle a \rangle}{\langle b \rangle}}\ba a}(X)$ if and only if
$b\in\cl_{M_{\langle b \rangle}}(X)$. Since $M_{{\langle a \rangle}{\langle b \rangle}}\ba a,b=M_{\langle b \rangle}\ba b$, this will prove that 
$M_{{\langle a \rangle}{\langle b \rangle}}\ba a=M_{\langle b \rangle}$.
Say $b\in\cl_{M_{{\langle a \rangle}{\langle b \rangle}}\ba a}(X)$. Then $b\in\cl_{M_{{\langle a \rangle}{\langle b \rangle}}}(X)$. Hence
$\lm_{{M_{\langle a \rangle}}/X}(B-X)=0$. But $M/X = M_{\langle a \rangle} /X \ba a$, so  $\lm_{M/X}(B-X)=0$. Thus $b\in\cl_{M_{\langle b \rangle}}(X)$.

Assume that $b\in\cl_{M_{\langle b \rangle}}(X)$. Then $\lm_{M/X}(B-X)=0$. Since $a\in\cl_{M_{\langle a \rangle}}(A)$ and 
$A\subseteq E(M)-B$, we have $a\in\cl_{M_{\langle a \rangle}/X}((E(M)-B)-X)$. It follows that 
$\lm_{{M_{\langle a \rangle}}/X}(B-X)=0$. Hence $b\in\cl_{M_{{\langle a \rangle}{\langle b \rangle}}}(X)$. Thus \ref{subdifferent-1} holds.

Assume that $M_{{\langle a \rangle}{\langle b \rangle}}\neq M_{{\langle b \rangle}{\langle a \rangle}}$, and that, amongst all counterexamples to the lemma, $|E(M)|$ is
minimal. Then there is a set $Z$ that is independent in one of $M_{{\langle a \rangle}{\langle b \rangle}}$ and $M_{{\langle b \rangle}{\langle a \rangle}}$, say the second, 
and is a circuit in the other, $M_{{\langle a \rangle}{\langle b \rangle}}$. By (\ref{subdifferent-1}), $\{a,b\}\subseteq Z$. This implies that neither $a$ nor $b$ is a 
loop of $M_{{\langle a \rangle}{\langle b \rangle}}$ or of $M_{{\langle b \rangle}{\langle a \rangle}}$. Hence $\lm_M(A),\lm_M(B)>0$. 

Let $Z'=Z-\{a,b\}$ and suppose $Z' \neq \emptyset$. In this case, it follows from Lemma~\ref{contract-ok}
that the triple $(M/Z',  M_{{\langle a \rangle}{\langle b \rangle}}/Z',M_{{\langle b \rangle}{\langle a \rangle}}/Z')$ also gives a counterexample to the theorem
contradicting the minimality of $|E(M)|$. Hence
$Z=\{a,b\}$. 

Let $C=E(M)-(A\cup B)$. Since $\{a,b\}$ is a circuit in $M_{{\langle a \rangle}{\langle b \rangle}}$, we have
$b\in\cl_{M_{{\langle a \rangle}{\langle b \rangle}}}(\{a\})$. This means that $\{a\}$ is in the modular cut that generates $M_{{\langle a \rangle}{\langle b \rangle}}$ from 
$M_{\langle a \rangle}$. Since $M_{{\langle a \rangle}{\langle b \rangle}}$ is obtained from $M_{\langle a \rangle}$ by freely adding $b$  into the guts of $B$,
we have $\lm_{M_{\langle a \rangle}/a}(B)=0$. We have observed that $\lm_M(B)>0$. We deduce that $(A\cup C,B)$ is a $2$-separation in $M$, so $M_{\langle a \rangle}$ is a parallel connection with basepoint $a$ of matroids with ground sets $A \cup C \cup a$ and $B \cup a$. Hence $a\in\cl_{M_{\langle a \rangle}}(A\cup C)$ and  $a\in \cl_{M_{\langle a \rangle}}(B)$. 

Since $M_{{\langle b \rangle}{\langle a \rangle}}\ba b=M_{\langle a \rangle}$, we have that $a\in\cl_{M_{{\langle b \rangle}{\langle a \rangle}}}(A\cup C)$ and that $a\in\cl_{M_{{\langle b \rangle}{\langle a \rangle}}}(B)$.
As  $M_{\langle b \rangle}$ is obtained from $M$ by freely adding $b$  into 
the guts of $B$, and $A \cup C = E(M) - B$, we have that $b\in\cl_{M_{{\langle b \rangle}}}(B)$ and $b\in\cl_{M_{{\langle b \rangle}}}(A \cup C)$. 
But $M_{{\langle b \rangle}{\langle a \rangle}}$ is an extension of $M_{\langle b \rangle}$. It follows that $b\in\cl_{M_{{\langle b \rangle}{\langle a \rangle}}}(B)$
and 
$b\in\cl_{M_{{\langle b \rangle}{\langle a \rangle}}}(A\cup C)$. 

Now $\sqcap_{M_{{\langle b \rangle}{\langle a \rangle}}}(A\cup C,B)=1$ and $\{a,b\}\subseteq \cl_{M_{{\langle b \rangle}{\langle a \rangle}}}(A\cup C)\cap \cl_{M_{{\langle b \rangle}{\langle a \rangle}}}(B)$.
Hence $\{a,b\}$ is dependent in $M_{{\langle b \rangle}{\langle a \rangle}}$, contradicting the assumption that this set is independent
in $M_{{\langle b \rangle}{\langle a \rangle}}$.
\end{proof}

\begin{lemma}
\label{guts-away}
Let $\{X_1,X_2,\ldots,X_n\}$ be a collection of disjoint sets in a matroid $M$ and let
$\{Y_1,Y_2,\ldots,Y_n\}$ be a collection of disjoint sets each of which is disjoint from $E(M)$.
Let $\phi$ be a permutation of $[n]$. Let $M_{\phi(0)}=M$, and, for each $i$ in $[n]$,
let $M_{\phi(i)}$ be the matroid that is obtained from $M_{\phi(i-1)}$ by freely adding the elements of
$Y_i$ into the guts of $X_i$. Let $M_\phi=M_{\phi(n)}$. Then the following hold.
\begin{itemize}
\item[(i)] If $\psi$ is also a permutation of $[n]$, then $M_\psi=M_\phi$. 
\item[(ii)] If $i \in [n]$, then $M_\phi$ is obtained from $M_\phi\ba Y_i$ by freely adding  the elements
of $Y_i$  into the guts of $X_i$.
\end{itemize}
\end{lemma}

\begin{proof}
Part (i) follows from Lemma~\ref{different-guts} and a routine induction. We omit the details. For (ii),
choose a permutation $\psi$ such that $\psi(i)=n$.
\end{proof}

\subsection*{The clonal core of a partitioned matroid} When $M$ is a matroid on a set $E$, and 
$\mathcal X$ is a partition $\{X_1,X_2,\ldots,X_n\}$ of $E$, recall that 
the pair $(M,\mathcal X)$ is a   partitioned matroid. 
%We note here that the partition $\mathcal X$ may include empty members. 
The partitioned matroid $(N, \{Y_1,Y_2,\dots,Y_n\})$ is {\em isomorphic} to the partitioned matroid $(M, \{X_1,X_2,\dots,X_n\})$ 
if there is a bijection $\varphi: E(M) \to E(N)$ such that $r_M(X) = r_N(\varphi(X))$ for all subsets $X$ of $E(M)$, and $\varphi(X_i) = Y_i$ for all $i$ in $[n]$.

We now describe a construction that 
builds an associated partitioned matroid $(\widehat M,\mathcal Y)$ from the partitioned matroid $(M,\mathcal X)$.
\begin{itemize}
\item[(i)] Let $\mathcal Y$  be a collection $\{Y_1,Y_2,\ldots,Y_n\}$ of disjoint sets each disjoint from $E$ such
that $|Y_i|=\lm_M(X_i)$ for each $i$ in $[n]$. 
\item[(ii)] Let $M_0=M$ and, for each $i$ in $[n]$, let $M_i$ be the matroid obtained from 
$M_{i-1}$ by freely adding the elements of $Y_i$ into the guts of $X_i$ in  $M_{i-1}$. 
\item[(iii)] Let $\widehat M=M_n\ba E$.
\end{itemize}

It follows from Lemma~\ref{guts-away} that the partitioned matroid $(\widehat M,\mathcal Y)$  does not depend on
the ordering of the members of $\mathcal X$. We say that $(\widehat M,\mathcal Y)$ is the {\em clonal core}
of $(M,\mathcal X)$. 
Note that there is no assumption here that $\lm(X_i)>0$. When $\lm(X_i)=0$, we have that 
$Y_i=\emptyset$. In particular, if some $X_i$ is a separator of $M$, then
 the clonal core of $M$ is obtained from  the clonal core of $M\ba X_i$ by adding an empty set $Y_i$ to $\mathcal Y$.

A major reason for the introduction of the clonal core is to enable us to infer certain connectivity properties of $M$ from the corresponding connectivity properties of $\widehat M$. The proof of the next result, though not deep, 
is long and technical. 
When $\{W_1,W_2,\dots,W_n\}$ is a family of subsets of a set $S$, and  $J$ is  a non-empty subset of $[n]$, we write $W_J$ 
for $\cup_{j \in J} W_j$; when $J$ is empty, $W_J = \emptyset$.

%In the next theorem, ${\widehat M}$ is constructed as described at the beginning of this subsection. 

\begin{theorem}
\label{G4} 
Let $(M,\mathcal X)$ be a partitioned matroid, where $\mathcal X=\{X_1,X_2,\dots,X_n\}$,
and let  $({\widehat M},\mathcal Y)$ be the clonal core of $(M, \mathcal X)$, where 
$\mathcal Y=\{Y_1,Y_2,\ldots, Y_n\}$. 
Then $|E({\widehat M})|= \lm_M(X_1) + \lm_M(X_2) + \dots + \lm_M(X_n)$.  Moreover, in 
$(\widehat M,\mathcal Y)$, 
\begin{itemize}
\item[(i)] each $Y_i$ consists of an independent, coindependent set of clones of cardinality $\lambda_M(X_i)$;  and 
\item[(ii)] for all non-empty disjoint subsets $J$ and $K$ of $[n]$,
\begin{itemize}
\item[(a)] $\lambda_M(X_J) = \lambda_{\widehat{M}}(Y_J)$; 
\item[(b)] $\sqcap_M(X_J,X_K) = \sqcap_{\widehat{M}}(Y_J,Y_K)$; and 
\item[(c)] $\sqcap^*_M(X_J,X_K) = \sqcap^*_{\widehat{M}}(Y_J,Y_K).$
\end{itemize}
\end{itemize}
\end{theorem}

\begin{proof}
Let $\lambda_M(X_i) = t_i$. Construct the matroid $M_{t_1}$ from $M$ by freely adding a $t_1$-element independent set $G_1$ of clones to the guts of 
$(X_1,X_2\cup X_3 \cup \dots \cup X_n)$ as in Theorem~\ref{first}. Let $N_1 = M_{t_1} \ba X_1$. By Lemma~\ref{G3}, the set $G_1$ is an independent, coindependent set of clones 
in $N_1$ and $\lambda_{N_1}(G_1) = \lambda_M(X_1)$. Then the ground set of $N_1$ has a partition into non-empty sets $G_1,X_2,X_3,\dots,X_n$. Moreover, by Lemma~\ref{G2}, 
for all disjoint subsets $J$ and $K$ of $\{2,3,\dots,n\}$, we have $\sqcap_M(X_J,X_K) = \sqcap_{N_1}(X_J,X_K)$ and 
$\sqcap_M(X_1 \cup X_J,X_K) = \sqcap_{N_1}(G_1 \cup X_J,X_K)$. Also, by Lemma~\ref{G2*}, 
 $\sqcap^*_M(X_J,X_K) = \sqcap^*_{N_1}(X_J,X_K)$ and 
$\sqcap^*_M(X_1 \cup X_J,X_K) = \sqcap^*_{N_1}(G_1\cup X_J,X_K)$.

Assume that $N_1,N_2,\dots,N_i$ have been defined so that $E(N_i)$ is the disjoint union of $G_1,G_2,\dots,G_i,X_{i+1},X_{i+2},\dots, X_n$ where 
\begin{itemize}
\item[(1)] for $j \le i$, 
each $G_j$ is an independent, coindependent set of clones of $N_i$ of cardinality $\lambda_M(X_j)$; and 
\item[(2)] for all disjoint subsets $I_1$ and $I_2$ of $\{1,2,\dots,i\}$ and all disjoint subsets $J_1$ and $J_2$ of 
$\{i+1,i+2,\dots,n\}$, 
\begin{equation}
\label{eq1}
\sqcap_M(X_{I_1} \cup X_{J_1}, X_{I_2} \cup X_{J_2}) = \sqcap_{N_i}(G_{I_1} \cup X_{J_1}, G_{I_2} \cup X_{J_2})
\end{equation}
and 
\begin{equation}
\label{eq2}
\sqcap^*_M(X_{I_1} \cup X_{J_1}, X_{I_2} \cup X_{J_2}) = \sqcap^*_{N_i}(G_{I_1} \cup X_{J_1}, G_{I_2} \cup X_{J_2}).
\end{equation}
\end{itemize}

To define $N_{i+1}$ from $N_i$, first extend the latter by an independent, coindependent set $G_{i+1}$ of clones added into the guts of 
$(X_{i+1},E(N_i) - X_{i+1})$ where $|G_{i+1}| = \lambda_M(X_{i+1})$. Let $N'_i$ be the resulting extension and let $N_{i+1} = N'_i \ba X_{i+1}$. Thus 
$$E(N_{i+1}) = G_1 \cup G_2 \cup \dots \cup G_{i+1} \cup X_{i+2} \cup \dots \cup X_n.$$
By Lemma~\ref{G3}, in $N_{i+1}$, the set $G_{i+1}$ is independent and coindependent. Moreover, $G_{i+1}$ is a set of clones of cardinality $\lambda_{N_i}(X_{i+1})$. 
By (\ref{eq1}), this cardinality is $\lambda_M(X_{i+1})$.

\begin{sublemma}
\label{sub1} 
For each $t$ in $\{1,2,\dots,i\}$, the set $G_t$ is an independent, coindependent set of clones of cardinality $\lambda_M(X_t)$ in $N_{i+1}$.
\end{sublemma}

To see this, first note that $|G_t| = \lambda_M(X_t)$. Moreover,  $G_t$ is an independent, coindependent set of clones in $N_t$. Thus $G_t$ is an 
independent set in $N_{i+1}$. Suppose that $N_{i+1}$ has a cyclic flat $F$ for which there are elements $x$ and $y$ of $G_t$ such that $x \in F$ and $y \not\in F$. 
Because $x$ and $y$ are clones in $N_{i+1} \ba G_{i+1}$, which equals $N_i\ba X_{i+1}$, the cyclic flat $F$ contains an element of $G_{i+1}$. Thus $F$ contains $G_{i+1}$. 
By 
Corollary~\ref{G2cor}, $(F- G_{i+1}) \cup X_{i+1}$ is a flat of $N_i$ that contains $x$ but not $y$. Thus $x$ must be a coloop of $N_{i}|((F - G_{i+1}) \cup X_{i+1})$. 
But $x$ is not a coloop of $N_{i+1}|F$, so $F$ contains a circuit $C$ containing $x$. Then $C$ meets $G_{i+1}$. Let $C \cap G_{i+1} = \{x_1,x_2,\dots,x_s\}$. 
Then $x_1 \in \cl_{N'_i}(X_{i+1})$, so there is a circuit $C_1$ such that $x_1 \in C_1 \subseteq X_{i+1} \cup \{x_1\}$. Thus there is a circuit $C'$ such that 
$x \in C' \subseteq (C \cup C_1) - \{x_1\}$. Hence $C' \cap G_{i+1} \subseteq \{x_2,x_3,\dots,x_s\}$. By repeatedly eliminating the elements of $C' \cap G_{i+1}$, 
we obtain the contradiction that $x$ is in a circuit that is contained in $(F- G_{i+1}) \cup X_{i+1}$. We conclude that $G_t$ is a set of clones in $N_{i+1}$.

Finally, from Lemma~\ref{G2}(ii), $\lambda_{N_{i+1}}(G_t) = \lambda_{N_{i}}(G_t) = |G_t|$, so $G_t$ is coindependent in $N_{i+1}$. Thus \ref{sub1} holds.

Now let $I_1$ and $I_2$ be disjoint subsets of $\{1,2,\dots,i\}$ and let $J_1$ and $J_2$ be disjoint subsets of $\{i+2,i+3,\dots,n\}$. Then, by (\ref{eq1}) and Lemma~\ref{G2}, we have 
\begin{eqnarray*}
\sqcap_M(X_{I_1 \cup J_1},X_{I_2 \cup J_2}) & = & \sqcap_{N_i}(G_{I_1} \cup X_{J_1}, G_{I_2} \cup X_{J_2})\\
&  = & \sqcap_{N_{i+1}}(G_{I_1} \cup X_{J_1}, G_{I_2} \cup X_{J_2}),
\end{eqnarray*}
 and 
\begin{eqnarray*}
\sqcap_M(X_{I_1 \cup J_1\cup \{i+1\}},X_{I_2 \cup J_2}) & = & \sqcap_{N_i}(G_{I_1} \cup X_{J_1 \cup \{i+1\}}, G_{I_2} \cup X_{J_2})\\
&  = & \sqcap_{N_{i+1}}(G_{I_1 \cup \{i+1\}} \cup X_{J_1}, G_{I_2} \cup X_{J_2}).
\end{eqnarray*}
Likewise, by (\ref{eq2}) and Lemma~\ref{G2*}, 
$$\sqcap^*_M(X_{I_1 \cup J_1},X_{I_2 \cup J_2})  = \sqcap^*_{N_{i+1}}(G_{I_1} \cup X_{J_1}, G_{I_2} \cup X_{J_2})$$
 and 
$$\sqcap^*_M(X_{I_1 \cup J_1\cup \{i+1\}},X_{I_2 \cup J_2})  = \sqcap^*_{N_{i+1}}(G_{I_1 \cup \{i+1\}} \cup X_{J_1}, G_{I_2} \cup X_{J_2}).$$

The theorem follows by taking $Y_i = G_i$ for all $i$ in $[n]$, noting that we get from Lemma~\ref{guts-away} that  $\widehat{M}$ is the matroid $N_n$ constructed  above.
\end{proof}

We recall that a partitioned matroid $(N, \{Z_1,Z_2,\dots,Z_n\})$ is a clonal-core matroid if, for each $i$ in $[n]$, the set $Z_i$ is a set of clones of cardinality $\lambda_N(Z_i)$. 
Observe that, because  $Z_i$ has cardinality $\lambda_N(Z_i)$, it follows that $Z_i$ is independent. Moreover, $Z_i$ is coindependent since $\lambda_N(Z_i) = \lambda_{N^*}(Z_i)$. 

%We say that a partitioned matroid $(N,\mathcal Z)$ is a {\em clonal-core matroid} if there is a partitioned matroid
%$(M,\mathcal X)$ with the property that, up to labels, $(N,\mathcal Z)$ is equal to the clonal core of 
%$(M,\mathcal X)$.  

\begin{lemma}
\label{cc}
A partitioned matroid $(N,\mathcal Z)$ is a clonal-core matroid if and only
if there is a partitioned matroid
$(M,\mathcal X)$ whose clonal core is isomorphic to $(N,\mathcal Z)$. 
\end{lemma}

\begin{proof}
If the clonal core of $(M,\mathcal X)$   is isomorphic to $(N,\mathcal Z)$, then, 
by Theorem~\ref{G4},   $(N,\mathcal Z)$ is a clonal-core matroid. Conversely, suppose 
 $(N,\mathcal Z)$ is a clonal-core matroid. We shall show that the clonal core   of  $(N,\mathcal Z)$ is isomorphic to $(N,\mathcal Z)$. 
Since  $(N,\mathcal Z)$ is a clonal-core matroid, if $Z \in \mathcal Z$, then $Z$ is a set of clones in $N$ of cardinality $\lambda_N(Z)$.  Let $\lambda_N(Z) = t$. 
We assume that $t> 0$ otherwise $Z$ is empty. Let $Z = \{f_1,f_2,\dots,f_t\}$ and let $N_t$ be the matroid that is obtained from $N$ by freely adding the set $\{e_1,e_2,\dots,e_t\}$ into the guts of 
$Z$. By Theorem~\ref{first}, $\{e_1,e_2,\dots,e_t\}$ is an independent set of clones of $N_t$ and $\cl_{N_t}(E(N) - Z) \cap \cl_{N_t}(Z)$ contains and is spanned by $\{e_1,e_2,\dots,e_t\}$.

We show next that 
\begin{sublemma}
\label{ccsub1} 
$\{e_1,e_2,\dots,e_t, f_1,f_2,\dots,f_t\}$ is a set of clones of $N_t$. 
\end{sublemma}

To see this, it suffices to show that each $f_i$ is a clone of $e_1$. Assume that this fails. Then $N_t$ has a cyclic flat $F$ that contains 
exactly one of $e_1$ and $f_i$. Suppose that $f_i \in F$ but $e_1 \not\in F$. As $e_1,e_2,\dots,e_t$ are clones in $N_t$, we deduce that $F \cap \{e_1,e_2,\dots,e_t\} = \emptyset$. Since $f_1,f_2,\dots,f_t$ are clones in $N$, it follows that $\{f_1,f_2,\dots,f_t\} \subseteq F$. As $e_1,e_2,\dots,e_t$ were added to the guts of $\{f_1,f_2,\dots,f_t\}$, it follows that $\{e_1,e_2,\dots,e_t\} \subseteq \cl_{N_t}(\{f_1,f_2,\dots,f_t\}) \subseteq F$, a contradiction. 

Now suppose that $e_1 \in F$ but $f_i \not\in F$. Then  $\{e_1,e_2,\dots,e_t\} \subseteq F$ as $e_1,e_2,\dots,e_t$ are clones in  $N_t$. Since $\{e_1,e_2,\dots,e_t\}$ spans $\{f_1,f_2,\dots,f_t\}$, we obtain a contradiction. Hence 
\ref{ccsub1} holds. 

By \ref{ccsub1}, $N_t \backslash\{ f_1,f_2,\dots,f_t\} \cong N$. By repeating this argument for each set $Z$ in $\mathcal Z$, we deduce that the clonal core   of  $(N,\mathcal Z)$ is isomorphic to $(N,\mathcal Z)$. Thus the lemma holds.
\end{proof}

Theorem~\ref{G4} provides conditions that enable us to say that, if a theorem 
on partitioned matroids holds for all clonal-core matroids, then it holds for all partitioned matroids.
The next, somewhat informal, corollary follows immediately from Theorem~\ref{G4}. Note that a precise formulation would
require an exercise in logic that we would prefer to avoid. 

\begin{corollary}
\label{core-blimey}
If a statement for partitioned matroids involves only
\begin{itemize}
\item[(i)] connectivities of unions of members of the partition,
\item[(ii)] local connectivities between disjoint pairs of unions of members of the partition, and
\item[(iii)] local coconnectivites between disjoint pairs of unions of members of the partition,
\end{itemize}
then the statement is true for all partitioned matroids if and only if it is true for all 
clonal-core matroids.
\end{corollary}

%For a matroid $M$ and a partition of $E(M)$ into sets $X_1,X_2,\dots,X_n$ for some $n \ge 2$, we call the matroid $\widehat{M}$ constructed from $M$ as in Lemma~\ref{G4} 
%the {\it clonal core} of $(M, \{X_1,X_2,\dots,X_n\})$ or, when the partition is clear, the clonal core of $M$. We saw above that the ground set of the clonal core of 
%$(M, \{X_1,X_2,\dots,X_n\})$ has $\lambda(X_1) + \lambda(X_2) + \dots  + \lambda(X_n)$ members. 

%%% QUESTION. Is the dual of the clonal core of a matroid, the clonal core of the dual matroid????

\section{The Behaviour  of $(4,c)$-Flexipaths}
\label{beh}

When we have  a matroid $M$ having a path 
$(L,Q_1,Q_2,\ldots,Q_n,R)$ of $4$-separations and $t \le n$, we can consider 
$$(L \cup Q_1\cup Q_2 \cup \dots \cup Q_j, Q_{j+1},Q_{j+2}, \ldots, Q_{j+t}, Q_{j+t+1}\cup Q_{j+t+2} \cup \dots \cup Q_n \cup R),$$
which is also a path $4$-separations, this one having exactly $t$ internal steps. 
Moreover, if the original path is a $4$-flexipath, so too is the second path.

Now let $(L,Q_1,Q_2,\ldots,Q_n,R)$ be a $4$-flexipath in a  matroid $M$. 
Because we are dealing with a flexipath, we may use the idea from the previous paragraph of absorbing internal steps into the end steps to assume that 
$\lm(Q_i) = \lm(Q_j)$ for all $i$ and $j$. By Lemma~\ref{mono}, for distinct $i$ and $j$, we have $\lm(Q_i \cup Q_j) \ge \lm(Q_i)$. If equality holds here, we may replace $Q_i$ and $Q_j$ by a new step, $Q_i \cup Q_j$. By repeating this process, we eventually obtain 
a $(4,c)$-flexipath for some $c$ in $\{1,2,3\}$, that is, $\lm(Q_i)=c$ for all $i$, and 
$\lm(Q_i\cup Q_j)>c$ for all distinct $i$ and $j$.

In this section, we shall derive some general properties of a $(4,c)$-flexipath $(L,Q_1,Q_2,\ldots,Q_n,R)$. 
We show in Lemma~\ref{GGGG} that we may assume that $c \le 3$ otherwise $n \le 1$. 
%Now let $(L,Q_1,Q_2,\ldots,Q_n,R)$ be a $(4,c)$-flexipath in a  matroid $M$. 
%Originally I had that M was required to be connected but I can no longer see why this is.
By Theorem~\ref{G4}, there is a matroid ${\widehat{M}}$ 
having $cn + 6$ elements whose ground set is the disjoint union of the sets $\wt{L},\wt{Q}_1,\wt{Q}_2,\dots,\wt{Q}_n,\wt{R}$ where, for each $i$ in $[n]$, the set $\wt{Q}_i$ 
is a $c$-element independent, coindependent set of clones, and each of $\wt{L}$ and $\wt{R}$ consists of a $3$-element independent, coindependent set of clones. 
Moreover, for all subsets $I$ of $\{1,2,\dots,n\}$, we have 
$\lambda_{\wt{M}}(\wt{L} \cup \wt{Q}_I)= 3$ and, for all disjoint subsets $I_1$ and $I_2$ of $[n]$,
$$\sqcap_M(Q_{I_1}, Q_{I_2}) = \sqcap_{\wt{M}}(\wt{Q}_{I_1}, \wt{Q}_{I_2})$$ 
and 
$$\sqcap_M(L \cup R \cup Q_{I_1}, Q_{I_2}) = \sqcap_{\wt{M}}(\wt{L} \cup \wt{R} \cup \wt{Q}_{I_1}, \wt{Q}_{I_2}).$$ 
In view of this, as noted in the previous section, we can infer much about the matroid $M$ by focusing on its clonal core $\wt{M}$. This means that, in many arguments, we will assume that $M$ is a clonal-core matroid corresponding to the partition $\{L,Q_1,Q_2,\ldots,Q_n,R\}$ of $E(M)$.

%It will be convenient then to consider a matroid $M$ having a path 
%$(L,Q_1,Q_2,\ldots,Q_n,R)$ of $4$-separations where each of $L$ and $R$ is a 3-element, independent, coindependent set of clones, and each 
%$Q_i$ is a 2-element, independent, coindependent set of clones. For such a matroid having such a path of $4$-separations, we can consider 
%$$(L \cup Q_1\cup Q_2 \cup \dots \cup Q_j, Q_{j+1},Q_{j+2}, \ldots, Q_{j+k}, Q_{j+k+1}\cup Q_{j+k+2} \cup \dots \cup Q_n \cup R),$$
%which is also a path $4$-separations, say $(L', Q_{j+1},Q_{j+2}, \ldots, Q_{j+k}, R')$. If we take the clonal core of $M$ with respect to this new partition of $E(M)$, 
%then we have reduced to a path with $k$ internal steps.

In the next section, we will focus on $(4,2)$-flexipaths. Before doing that, we develop some general results for $(4,c)$-flexipaths. In all of the results in this section, $(L,Q_1,Q_2,\ldots,Q_n,R)$ is a $(4,c)$-flexipath in a matroid $M$. The main results of this section, Corollary~\ref{414} and Theorem~\ref{433}, determine  all possible $(4,1)$-flexipaths and all possible $(4,3)$-flexipaths, respectively. Each of the latter has at most two internal steps. 
%The reason for focusing on $(4,2)$-flexipaths is that it is within such paths that the most complex and varied structures arise.

\begin{lemma}
\label{B} 
For all $i$ in $[n]$,
$$\sqcap(L,Q_i) + \sqcap^*(L,Q_i) = c = \sqcap(R,Q_i) + \sqcap^*(R, Q_i).$$
\end{lemma}

\begin{proof} 
By symmetry, it suffices to prove the first equality. 
Using Lemma~\ref{geoffs}, we have 
\begin{eqnarray*}
3 = \lm(L \cup Q_i) & = & \lm(L) + \lm(Q_i) - \sqcap(L,Q_i)  - \sqcap^*(L,Q_i)\\
& = & 3 + c - \sqcap(L,Q_i)  - \sqcap^*(L,Q_i).
\end{eqnarray*}
\end{proof}

\begin{lemma}
\label{C} 
For all distinct $i$ and $j$ in $[n]$,
$$\sqcap(Q_i,Q_j) + \sqcap^*(Q_i,Q_j) \le  c- 1.$$
\end{lemma}

\begin{proof} 
Using  Lemma~\ref{geoffs}, we have 
\begin{eqnarray*}
c+1 \le  \lm(Q_i \cup Q_j) & = & \lm(Q_i) + \lm(Q_j) - \sqcap(Q_i,Q_j)  - \sqcap^*(Q_i,Q_j)\\
& = & c + c - \sqcap(Q_i,Q_j)  - \sqcap^*(Q_i,Q_j).
\end{eqnarray*}
The result follows immediately.
\end{proof}

%As we noted at the end of the last section, we can learn much about the connectivity and local connectivity functions in $M$  by considering the clonal core of $(M,\{L,Q_1,Q_2,\ldots,Q_n,R\})$. In this clonal core, each of $L$ and $R$ is replaced by an independent, coindependent set of clones of size three, while each $Q_i$ is replaced by 
%an independent, coindependent set of clones of size $c$.  
In each of the remaining proofs in this section, by relying on Theorem~\ref{G4}, we may assume that $(M,\{L,Q_1,Q_2,\ldots,Q_n,R\})$ is a clonal-core matroid. 
%argue in the clonal core of $(M,\{L,Q_1,Q_2,\ldots,Q_n,R\})$ to obtain the result for $(M,\{L,Q_1,Q_2,\ldots,Q_n,R\})$ itself. When we do this, to simplify the notation, we will denote this clonal core by $(M,\{L,Q_1,Q_2,\ldots,Q_n,R\})$ rather than by $(\wt{M}, \{\wt{L},\wt{Q_1},\wt{Q_2},\ldots,\wt{Q_n},\wt{R}\})$. 

\begin{lemma}
\label{A}
For all $i$ in $[n]$,  
$$\sqcap(L,Q_i) = \sqcap(R,Q_i).$$
\end{lemma}

 \begin{proof} 
 Let $(L,Q,R)$ be a path of $4$-separations. Then, by Lemma~\ref{jamess}(ii),
 \begin{eqnarray*}
\sqcap(L,Q) & = & \sqcap(R,Q)+ \sqcap(R \cup Q, L) - \sqcap(L \cup Q, R)\\
& = & \sqcap(R,Q) + \lm(L) - \lm(R)\\
& = &\sqcap(R,Q).
\end{eqnarray*}
In particular,  the lemma holds when $n= 1$. 
 
Assume $n \ge 2$. Because we are dealing with a flexipath, we may assume that $i = 1$. Then 
$(L,Q_1,Q_2\cup \dots \cup Q_n\cup R)$ is a path of $4$-separations, so 
 $$\sqcap(L,Q_1) = \sqcap(Q_1, Q_2 \cup Q_3 \cup \dots \cup Q_n \cup R) \ge \sqcap(Q_1, R),$$
where the inequality follows by the monotonicity of $\sqcap$ in each argument. 
By symmetry, $\sqcap(R,Q_1)  \ge \sqcap(Q_1, L)$ so $\sqcap(L,Q_1) = \sqcap(R,Q_1).$
\end{proof}

\begin{lemma}
\label{GG} 
If $n \ge 2$, then $\sqcap(L,R) + \sqcap^*(L,R) \le 5 - c.$
\end{lemma}

\begin{proof}
We have, by Lemma~\ref{geoffs}, that 
$$\sqcap(L,R) + \sqcap^*(L,R) = \lm(L) + \lm(R) - \lm(L \cup R) = 3+3 - \lm(L \cup R).  $$
As $\lm(L \cup R) = \lm(Q_1 \cup Q_2 \cup \dots \cup Q_n) \ge c+1$, the lemma follows.
\end{proof}

\begin{lemma}
\label{GGG} 
If $n \ge 3$, then $c \le 2$.
\end{lemma}

\begin{proof} 
Because $(L, Q_1,Q_2,\dots,Q_n,R)$ is a $(4,c)$-flexipath, so is $(L \cup Q_1,Q_2,\dots,Q_n,R)$. 
Thus, by Lemma~\ref{GG}, 
$\sqcap(L \cup Q_1,R) + \sqcap^*(L \cup Q_1,R) \le 5 - c.$
Therefore, by Lemma~\ref{B} and monotonicity, 
$$c = \sqcap(Q_1,R) + \sqcap^*(Q_1,R) \le 5 - c,$$
and the lemma follows.
\end{proof}

\begin{lemma}
\label{GGGG} 
If $c \ge 4$, then $n \le 1$.
\end{lemma}

\begin{proof}
Assume $n \ge 2$. By Lemma~\ref{GGG}, $n=2$. By Lemma~\ref{GG}, 
$$ 0 \le \sqcap(L,R) + \sqcap^*(L,R) \le 5-c \le 1.$$
Thus $c \in \{4,5\}$.

Suppose $\sqcap(L,Q_i) = 3$. Since we are operating in the clonal core, this means that $L \subseteq \cl(Q_i)$. 
By Lemma~\ref{A}, $R \subseteq \cl(Q_i)$, so $L \cup R \subseteq \cl(Q_i)$ and $r(L \cup R) \le r(Q_i) = c$. Thus 
$$6 - \sqcap(L,R) \le c.$$
This contradicts Lemma~\ref{GG}. Thus
$$\sqcap(L,Q_i) \le 2.$$

By duality, $\sqcap^*(L,Q_i) \le 2$. By Lemma~\ref{B}, 
$$c = \sqcap(L,Q_i) + \sqcap^*(L,Q_i) \le 4,$$
so $$c=4,$$
and, for each $i$ in $\{1,2\}$.
$$\sqcap(L,Q_i) = 2 = \sqcap^*(L,Q_i).$$

We deduce that, for each $N$ in $\{M, M^*\}$ and each $i$ in $\{1,2\}$, 
$$r_N(L \cup Q_i) = 5 = r_N(R \cup Q_i).$$ Thus, by the submodularity of $r_N$, 
$$r_N(L \cup R \cup Q_i) + r_N(Q_i) \le r_N(L \cup Q_i) + r_N(R\cup Q_i) = 10.$$ 
Hence $r_N(L\cup R \cup Q_i) \le 6.$  Therefore, by submodularity again, 
$$12\ge r_N(L \cup R \cup Q_1) + r_N(L \cup R \cup Q_2) \ge r_N(L \cup R) + r(N).$$
Taking each $N$ in $\{M,M^*\}$, we have 
$$24 \ge r_M(L \cup R) + r(M) + r_{M^*}(L \cup R) + r(M^*).$$ 
But $r(M) + r(M^*) = |E(M)| = 14.$
Thus $10 \ge r_M(L \cup R) + r_{M^*}(L \cup R)$, so 
$$10 \ge r(L) + r(R) +r^*(L) + r^*(R) - \sqcap(L,R) - \sqcap^*(L,R).$$
Hence  $\sqcap(L,R) + \sqcap^*(L,R) \ge 2$, which contradicts Lemma~\ref{GG}.
\end{proof}

The rest of this section is concerned with determining all possible $(4,1)$- and $(4,3)$-flexipaths beginning with the former. For such a flexipath, $\lm(Q_i) = 1$ for each $i$, so, since $M$ is a clonal-core matroid, we may take $Q_i = \{e_i\}$. 

\begin{lemma}
\label{411} 
Let $(L,e_1,e_2,\dots,e_n,R)$ be a $(4,1)$-flexipath  with $\sqcap(L,R) = 2$ and $n \ge 1$. 
Suppose $\sqcap(L,e_i) = 1$ for each $i$ in $\{1,2,\dots,t\}$ and 
$\sqcap(L,e_j) = 0$ for each $j$ in $\{t+1,t+2,\dots, n\}$. Then 
$\min\{t,n-t\} = 1$. 
\end{lemma}

\begin{proof} By Lemma~\ref{B}, 
$\sqcap^*(L,e_i) = 0$ for each $i$ in $\{1,2,\dots,t\}$ and $\sqcap(L,e_j) = 1$ for each $j$ in $\{t+1,t+2,\dots, n\}$. 
Suppose $n-t = 0$. Then 
$$3 = \lm(L \cup e_1 \cup  \dots \cup e_n) = r(L \cup e_1 \cup  \dots \cup e_n) + r(R)  - r(M) \le 3+3 - 4,$$
a contradiction. Thus $n - t > 0$. By duality, $t > 0$. 

Assume $t, n-t \ge 2$. By moving to the clonal core of the  $(4,1)$-flexipath 
$(L\cup e_1\cup \dots \cup e_{t-2},e_{t-1},e_t, e_{t+1},e_{t+2}, e_{t+3} \cup \dots \cup e_n\cup R)$, we may assume that $t= n-t = 2$. 
Since $\lm(\{e_i,e_j\}) = 2$ for $i \neq j$, we deduce that $r(\{e_i,e_j\}) = 2 = r^*(\{e_i,e_j\})$. As $\sqcap^*(L,e_3) = 1 = \sqcap^*(L,e_4)$, we see that $r^*(L \cup \{e_3,e_4\}) = r^*(L) = 3$. It follows by symmetry that 
 \begin{eqnarray*}
2+r^*(L \cup R) & = & r^*(\{e_3,e_4\}) + r^*(L \cup R \cup \{e_3,e_4\})\\
& \le & r^*(L \cup  \{e_3,e_4\}) + r^*(R \cup  \{e_3,e_4\})\\
& = &3+ 3.
\end{eqnarray*}
Thus $r^*(L \cup R) \le 4$, so $\sqcap^*(L,R) \ge 2$. But $\sqcap(L,R) = 2$ and, by Lemma~\ref{GG}, 
$\sqcap(L,R) + \sqcap^*(L,R) \le 4$, so $\sqcap^*(L,R) = 2$. This means that we can make inferences about $M^*$ from what we determine about $M$. 

Since $\sqcap^*(L,R) = 2$, we see that $r^*(L\cup R) = 4$. Thus 
$$4 = r^*(L \cup R \cup \{e_3,e_4\}) = |L \cup R \cup \{e_3,e_4\}| + r(\{e_1,e_2\}) - r(M),$$ so $r(M) = 6$. 
Dually, $r^*(M) = 6$, so $|E(M)| = 12$, a contradiction.
\end{proof}

We remind the reader that each of the matroids $M$ that arises in our lemmas is the clonal-core matroid associated with the $(4,c)$-flexipath $(L,Q_1,Q_2,\dots,Q_n,R)$.

\begin{lemma}
\label{413} 
Let $(L,e_1,e_2,\dots,e_n,R)$ be a $(4,1)$-flexipath  in a matroid $M$ with $\sqcap(L,R) = 3$. Then $r(M) = 3$ and 
$M|\{e_1,e_2,\dots,e_n\}$ can be any $n$-element simple matroid of rank at most three. 
\end{lemma}

\begin{proof}
First we show that 
\begin{sublemma}
\label{413sub}
$e_i \in \cl(L)$ for all $i$ in $[n]$.
\end{sublemma}

We observe that it suffices to show that $\sqcap(L,e_1) = 1$. Assume that $\sqcap(L,e_1) = 0$. Then $e_1 \notin \cl(L)$. Consider the 
$(4,1)$-flexipath $(L,e_1,\{e_2,e_3,\dots,e_n\} \cup R)$ rewriting this as $(L,e_1,R')$. As $r(L) = 3$, we have
$3 \ge \sqcap(L,R') \ge \sqcap(L,R) = 3$ so $\sqcap(L,R') = 3$. Recalling that  $M$ is a clonal-core matroid, 
we now consider  the clonal core of $(M,\{L,e_1,R'\})$ denoting this clonal core by $(\wt{M}, \{L,e_1,\wt{R'}\})$ 
rather than by $(\wt{M}, \{\wt{L},\wt{e_1},\wt{R'}\})$ because $L$ is already an independent, coindependent set of clones and $e_1$ is a singleton. 
Then $(L,e_1,\wt{R'})$ is a $(4,1)$-flexipath and $e_1 \notin \cl_{\wt{M}}(L)$. But $\cl_{\wt{M}}(L) = \cl_{\wt{M}}(L\cup \wt{R'})$, so $e_1 \notin \cl_{\wt{M}}(L\cup \wt{R'})$. This is a contradiction as it means $e_1$ is a coloop of $\wt{M}$, so $\lm_{\wt{M}}(\{e_1\}) = 0$. Thus \ref{413sub} holds. 

By \ref{413sub}, it follows that $r(M) = 3$. Clearly $r(M|\{e_1,e_2,\dots,e_n\}) \le 3$.  Now, let $N$ be any simple matroid of rank at most three with ground set $\{e_1,e_2,\dots,e_n\}$ and  take $M_0$ to be a copy of $U_{3,6}$ with ground set $\{f_1,f_2,\dots,f_6\}$ where 
$\{e_1,e_2,\dots,e_n\} \cap \{f_1,f_2,\dots,f_6\} = \emptyset$. Then, in  
the truncation to rank three of the direct sum of $M_0$ and $N$, we see that $(\{f_1,f_2,f_3\},e_1,e_2,\dots,e_n, \{f_4,f_5,f_6\})$ is a $(4,1)$-flexipath.
\end{proof}

Extending the last two lemmas, we get the following characterization, up to duality, of the clonal cores of all $(4,1)$-flexipaths. 

\begin{corollary}
\label{414} 
Let $(L,e_1,e_2,\dots,e_n,R)$ be a $(4,1)$-flexipath  in a matroid $M$ with $\sqcap(L,R) \ge \sqcap^*(L,R)$. Then 
one of the following holds.
\begin{itemize}
\item[(i)] $\sqcap(L,R) = 3$ and $r(M) = 3$, while  
$M|\{e_1,e_2,\dots,e_n\}$ is any $n$-element simple matroid of rank at most three, and $\{e_1,e_2,\dots,e_{n}\} \subseteq \cl(L) \cap \cl(R)$. 
\item[(ii)] $\sqcap(L,R) = 2$ and $r(M) = 4$, where $n \ge 2$ and, for some relabelling of  $\{e_1,e_2,\dots,e_{n}\}$, the matroid  
$M|\{e_1,e_2,\dots,e_{n-1}\}$ is simple and uniform of rank at most two where $\{e_1,e_2,\dots,e_{n-1}\} = \cl(L) \cap \cl(R)$ and 
$\{e_n\} = \cl^*(L) \cap \cl^*(R)$. 
%\item[(iii)] $\sqcap(L,R) = 1$ and $r(M) = 5$ where $n = 2$ and $\cl(L) \cap \cl(R) = \{e_1\}$ and $\cl^*(L) \cap \cl^*(R) = \{e_2\}$ for some relabelling of  $\{e_1,e_2\}$. 
\end{itemize}
\end{corollary}

\begin{proof}
By Lemma~\ref{GG}, $\sqcap(L,R) + \sqcap^*(L,R) \le 4$ provided $n \ge 2$. In Lemmas~\ref{411} and \ref{413}, we treated the cases where 
 $\sqcap(L,R) \in \{2,3\}$. By Lemmas~\ref{B} and \ref{A}, each $e_i$ is in exactly one of  $\cl(L) \cap \cl(R)$ and $\cl^*(L) \cap \cl^*(R)$. Since $\sqcap(L,R) \ge \sqcap^*(L,R)$, we see that $|\cl^*(L) \cap \cl^*(R)| \le 1$. If $\sqcap(L,R) = 3$, then (i) holds. If $\sqcap(L,R) = 2$, then $r(M) \ge 4$. By Lemma~\ref{411}, $|\cl^*(L) \cap \cl^*(R)| = 1$ and  (ii) holds.
\end{proof}

\begin{lemma}
\label{431}
In a $(4,3)$-flexipath with $n = 2$, 
$$(\sqcap(L,R),\sqcap^*(L,R)) = (\sqcap^*(Q_1,Q_2) + 6 - r(M), \sqcap(Q_1,Q_2) + 6 - r^*(M)).$$ 
\end{lemma}

\begin{proof}
By duality, it suffices to prove that the first coordinates are equal. We have 
 \begin{eqnarray*}
\sqcap^*(Q_1,Q_2) & = & r^*(Q_1) + r^*(Q_2) - r^*(Q_1 \cup Q_2)\\
& = & 3 + 3 - r^*(Q_1 \cup Q_2)\\
& = & r(L) + r(R) - (|Q_1 \cup Q_2| + r(L \cup R) - r(M))\\
& = &\sqcap(L,R) + r(M) -6.
\end{eqnarray*}
\end{proof}

\begin{lemma}
\label{431.5}
In a $(4,3)$-flexipath with $n = 2$, if $\sqcap(L,Q_1) = 2$, then $r(M) \le 5$. 
\end{lemma}

\begin{proof}
By Lemma~\ref{A},  $\sqcap(R,Q_1) = 2$, so 
 \begin{eqnarray*}
4 + 4 & = & r(L \cup Q_1) + r(R \cup Q_1)\\
& \ge & r(Q_1) + r(L \cup R \cup Q_1))\\
& = & 3 + r(M)
\end{eqnarray*}
as $Q_1$ is independent and $Q_2$ is coindependent. Thus $r(M) \le 5$.
\end{proof}

\begin{lemma}
\label{432}
In a $(4,3)$-flexipath with $n = 2$, for some $N$ in $\{M,M^*\}$ and each $i$ in $\{1,2\},$
$$(\sqcap_N(L,Q_i),\sqcap_N^*(L,Q_i)) = (2,1).$$ 
\end{lemma}

\begin{proof} 
By Lemma~\ref{B}, for each $N$ in $\{M,M^*\}$, we have 
$\sqcap_N(L,Q_i)+ \sqcap_N^*(L,Q_i) = 3.$ If $(\sqcap_N(L,Q_1),\sqcap_N^*(L,Q_1)) = (2,1)$ and 
$(\sqcap_N(L,Q_2),\sqcap_N^*(L,Q_2)) = (1,2)$, then, by Lemma~\ref{A}, we see that  $(\sqcap_N(R,Q_2),\sqcap_N^*(R,Q_2)) = (1,2)$. 
Thus $$3 = \lm_N(L \cup Q_1) = r_N(L \cup Q_1) + r_N(R \cup Q_2) - r(N) = 4 + 5 - r(N),$$ 
so $r(N) = 6$. As $|E(M)| = 12$, it follows that    
 $r(M) = r^*(M) = 6$.  But, by Lemma~\ref{431.5} and its dual, $r(M) \le 5$ and $r^*(M) \le 5$, a contradiction. 
 We deduce that $(\sqcap_N(L,Q_1), \sqcap^*_N(L,Q_1)) = (\sqcap_N(L,Q_2), \sqcap^*_N(L,Q_2))$ and the lemma follows. 
\end{proof}

The next theorem determines, up to duality,  the possible clonal cores of all $(4,3)$-flexipaths with at least two internal steps. 

\begin{theorem}
\label{433}
Consider a $(4,3)$-flexipath with $n \ge 2$ and $\sqcap(L,R) \ge \sqcap^*(L,R)$. Then $n= 2$ and  
one of the following three possibilities arises: 
\begin{itemize}
\item[(i)] $(\sqcap(L,R), \sqcap^*(L,R)) = (2,0)$ and $(\sqcap(Q_1,Q_2), \sqcap^*(Q_1,Q_2)) = (1,1)$;
\item[(ii)] $(\sqcap(L,R), \sqcap^*(L,R)) = (1,0) = (\sqcap(Q_1,Q_2), \sqcap^*(Q_1,Q_2))$; or 
\item[(iii)] $(\sqcap(L,R), \sqcap^*(L,R)) = (1,1)$ and $(\sqcap(Q_1,Q_2), \sqcap^*(Q_1,Q_2)) = (2,0)$.
\end{itemize}
\end{theorem}

\begin{proof} 
By Lemma~\ref{GGG}, since we are dealing with a $(4,3)$-flexipath, $n \le 2$, so $n = 2$. 
By Lemma~\ref{432}, for some $N$ in $\{M,M^*\}$, we have $\sqcap_N(L,Q_1) = 2$ and $\sqcap_N(R,Q_2) = 2$. 
Thus $r_N(L \cup Q_1) = 4 = r_N(R \cup Q_2)$. Hence 
$$3 = \lm_N(L \cup Q_1) = r_N(L\cup Q_1) + r_N(R \cup Q_2) - r(N) = 4 + 4 - r(N),$$ 
so $r(N) = 5$. As $|E(N)| = 12$, we see that $r^*(N) = 7$. It follows by Lemma~\ref{431} that 
\begin{equation}
\label{432+}
(\sqcap_N(L,R),\sqcap_N^*(L,R)) = (\sqcap_N^*(Q_1,Q_2) + 1, \sqcap_N(Q_1,Q_2) - 1),
\end{equation}
so  
$$\sqcap_N(L,R)+ \sqcap_N^*(L,R) = \sqcap_N(Q_1,Q_2) + \sqcap_N^*(Q_1,Q_2).$$
By Lemma~\ref{GG}, $\sqcap_N(L,R)+ \sqcap_N^*(L,R)\le 2.$  It follows by (\ref{432+}) that 
$1 \le \sqcap_N(L,R) \le 2$ and $1 \ge \sqcap_N^*(L,R) \ge 0$. As $\sqcap_M(L,R) \ge  \sqcap_M^*(L,R)$,  we deduce that 
$$(\sqcap_M(L,R),\sqcap_M^*(L,R)) \in \{(2,0), (1,0), (1,1)\}.$$
Moreover, by (\ref{432+}) again, we get $(\sqcap_M(Q_1,Q_2), \sqcap_M^*(Q_1,Q_2))$ for each of the three cases.
\end{proof}

To provide an example of a matroid satisfying Theorem~\ref{433}(ii), take a basis $\{b_1,b_2,b_3,b_4,p\}$ of $V(5, \mathbb{R})$. For each $i$ in $\{1,2,3,4\}$, freely add 
 points $x_{i,i+1}$, $y_{i,i+1}$, and $z_{i,i+1}$ to the plane spanned by $\{b_i,b_{i+1},p\}$ where $i+1$ is calculated modulo $4$. 
Delete $\{b_1,b_2,b_3,b_4,p\}$; let 
$$(L,R) = (\{x_{1,2}, y_{1,2}, z_{1,2}\},  \{x_{3,4}, y_{3,4}, z_{3,4}\})$$
and 
$$(Q_1,Q_2) = (\{x_{2,3}, y_{2,3}, z_{2,3}\},  \{x_{4,1}, y_{4,1}, z_{4,1}\}).$$

Next we construct  an example of a matroid satisfying Theorem~\ref{433}(i). Observe that we can determine the ranks of all subsets of $\{L,R,Q_1,Q_2\}$. Thus we can routinely check that if $X$ and $Y$ are such subsets, then  $r(X) + r(Y) \ge r(X \cup Y) + r(X \cap Y)$. It follows that we have a rank-$5$ polymatroid $P$ on the set of subsets of $\{L,R,Q_1,Q_2\}$. The operation of freely adding an element to a flat of a matroid extends straightforwardly to polymatroids (see, for example, \cite[p.409]{oxbook}). Using this idea, we   freely add  three   elements to each of $L$, $R$, $Q_1$, and $Q_2$. Then restricting the resulting polymatroid to the set of these twelve newly added points, we get a matroid satisfying Theorem~\ref{433}(i). Finally, to get a matroid corresponding to (iii) of the theorem, we can modify the example just given by interchanging $L$ with $Q_1$ and interchanging $R$ with $Q_2$. This switch does indeed produce an example because, in (i), we had $(\sqcap(Q_1,Q_2), \sqcap^*(Q_1,Q_2)) = (1,1)$ and 
$\sqcap(L,Q_1) = \sqcap(R,Q_1) = 2 = \sqcap(L,Q_2) = \sqcap(R,Q_2).$

\section{Types of $(4,2)$-Flexipaths}

The purpose of this section is to prove the main result of the paper, Theorem~\ref{flexi-types}, which describes all possible $(4,2)$-flexipaths. Let $\QQ$ be  a $(4,2)$-flexipath $(L,Q_1,Q_2,\ldots,Q_n,R)$ in a matroid $M$. 
In the introduction, we identified four special types of $(4,2)$-flexipaths, namely, spike-reminiscent, paddle-reminiscent, squashed, and stretched. Moreover, we noted that 
$\QQ$ is spike-reminiscent in $M$ if and only if it is paddle-reminiscent in $M^*$; and  $\QQ$ is squashed in $M$ if and only if $\QQ$ is stretched in $M^*$. 
Each of the remaining seven types of $(4,2)$-flexipaths has exactly three internal steps.

The flexipath $\QQ$ is {\em relaxed-spike-reminiscent} if all of the following hold:
\begin{itemize}
\item[(i)] $n=3$;
\item[(ii)] $\sqcap(L,R)=0$ and $\sqcap^*(L,R)=2$;
\item[(iii)] $\sqcap(Q_i,Q_j)=1$ and $\sqcap^*(Q_i,Q_j)=0$ for all distinct $i$ and $j$ in $[n]$; and 
\item[(iv)] $\sqcap(Q_i,L)=\sqcap(Q_i,R)=1 = \sqcap^*(Q_i,L)=\sqcap^*(Q_i,R)$
for all $i$ in $[n]$.
\end{itemize}

The flexipath $\QQ$ is {\em relaxed-paddle-reminiscent} if all of the following hold:
\begin{itemize}
\item[(i)] $n=3$;
\item[(ii)] $\sqcap(L,R)=2$ and $\sqcap^*(L,R)=0$;
\item[(iii)] $\sqcap(Q_i,Q_j)=0$ and $\sqcap^*(Q_i,Q_j)=1$  for all distinct $i$ and $j$ in $[n]$; and 
\item[(iv)] $\sqcap(Q_i,L)=\sqcap(Q_i,R)=1 = \sqcap^*(Q_i,L)=\sqcap^*(Q_i,R)$
for all $i$ in $[n]$.
\end{itemize}

Note that  $\QQ$ is relaxed-spike-reminiscent in $M$ if and only if it is relaxed-paddle-reminiscent in $M^*$. 

The flexipath 
$\QQ$ is {\em prism-like} if all of the following hold:
\begin{itemize}
\item[(i)] $n=3$;
\item[(ii)] $\sqcap(Q_i,Q_j)=\sqcap^*(Q_i,Q_j)=0$ for all distinct $i$ and $j$ in $[n]$;
\item[(iii)] $\sqcap(L,R)=\sqcap^*(L,R)=0$; and 
\item[(iv)] $\sqcap(Q_i,L)=\sqcap(Q_i,R)=1 = \sqcap^*(Q_i,L)=\sqcap^*(Q_i,R)$
for all $i$ in~$[n]$.
\end{itemize}
Observe that $\QQ$ is prism-like in $M$ if and only if $\QQ$ is prism-like in $M^*$. A diagram representing a rank-$6$ matroid with a prism-like flexipath is shown in Figure~\ref{prism}.

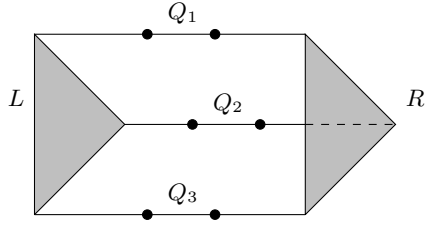
\begin{figure}
\center
\begin{tikzpicture}[scale=0.6]
%\draw[help lines] (-1.2, -1.2) grid (9.2, 6.2);

\draw[fill=lightgray] (0, 1) -- (2, 3) -- (0, 5) -- (0, 1);
\draw[fill=lightgray] (6, 1) -- (8, 3) -- (6, 5) -- (6, 1);
\draw (0, 5) -- node[above, xshift=5] {\footnotesize{$Q_1$}} (6, 5);
\draw (2, 3) -- node[above, xshift=5] {\footnotesize{$Q_2$}} (6, 3);
\draw[dashed] (6, 3) -- (8, 3);
\draw (0, 1) -- node[above, xshift=5] {\footnotesize{$Q_3$}} (6, 1);
\filldraw (2.5, 5) circle (3pt);
\filldraw (4, 5) circle (3pt);
\filldraw (3.5, 3) circle (3pt);
\filldraw (5, 3) circle (3pt);
\filldraw (2.5, 1) circle (3pt);
\filldraw (4, 1) circle (3pt);
\node[left, yshift=10] at (0, 3) {\footnotesize{$L$}};
\node[right, yshift=10] at (8, 3) {\footnotesize{$R$}};
\end{tikzpicture}
\caption{A prism-like flexipath $(L, Q_1, Q_2, Q_3, R)$.}
\label{prism} 
\end{figure}

The flexipath 
$\QQ$ is {\em tightened-prism-like} if all of the following hold.
\begin{itemize}
\item[(i)] $n=3$;
\item[(ii)] $\sqcap(Q_i,Q_j)=\sqcap^*(Q_i,Q_j)=0$ for all distinct $i$ and $j$ in $\{1,2,3\}$;
\item[(iii)] $\sqcap(L,R)=0$ and $\sqcap^*(L,R)=1$; and 
\item[(iv)] $\sqcap(Q_i,L)=\sqcap(Q_i,R)=1 = \sqcap^*(Q_i,L)=\sqcap^*(Q_i,R)$
for all $i$ in~$\{1,2,3\}$.
\end{itemize}
Note that we have not formally named what 
$\QQ$ is  in $M^*$ when $\QQ$ is tightened-prism-like in $M$.

The flexipath 
$\QQ$ is {\em doubly-tightened-prism-like} if all of the following hold.
\begin{itemize}
\item[(i)] $n=3$;
\item[(ii)] $\sqcap(Q_i,Q_j)=\sqcap^*(Q_i,Q_j)=0$ for all distinct $i$ and $j$ in $\{1,2,3\}$;
\item[(iii)] $\sqcap(L,R)=1 = \sqcap^*(L,R)$; and 
\item[(iv)] $\sqcap(Q_i,L)=\sqcap(Q_i,R)=1 = \sqcap^*(Q_i,L)=\sqcap^*(Q_i,R)$
for all $i$ in~$\{1,2,3\}$.
\end{itemize}
We see that $\QQ$ is doubly-tightened-prism-like in $M$ if and only if $\QQ$ is doubly-tightened-prism-like in $M^*$.

The flexipath $\QQ$ is {\it V\'{a}mos-inspired} if, in either $M$ or $M^*$, all of the following hold.
\begin{itemize}
\item[(i)] $n=3$;
\item[(ii)] $\sqcap(L,R)=0$ and  $\sqcap^*(L,R) =  1$;
\item[(iii)] $\sqcap(Q_i,L)=\sqcap(Q_i,R) = 1 =\sqcap^*(Q_i,L)=\sqcap^*(Q_i,R)$ for all $i$ in $\{1,2,3\}$; 
\item[(iv)] $\sqcap^*(Q_i,Q_j) = 0$ for all distinct $i$ and $j$;
and 
\item[(iv)] after a possible permutation of $\{1,2,3\}$, 
$$\sqcap(Q_1, Q_2) = 0 = \sqcap(Q_1, Q_3) \text{~~and~~} \sqcap(Q_2, Q_3) = 1.$$
\end{itemize}
Note that, by definition, $\QQ$ is V\'{a}mos-inspired in $M$ if and only if $\QQ$ is V\'{a}mos-inspired in $M^*$.

The flexipath $\QQ$ is {\it nasty} if  all of the following hold.
\begin{itemize}
\item[(i)] $n=3$;
\item[(ii)] $\sqcap(L,R)=1 = \sqcap^*(L,R)$;
\item[(iii)] $\sqcap(Q_i,L)=\sqcap(Q_i,R) = 1 =\sqcap^*(Q_i,L)=\sqcap^*(Q_i,R)$ for all $i$ in $\{1,2,3\}$; and 
\item[(iv)] after a possible permutation of $\{1,2,3\}$, 
$$
\begin{bmatrix}
\sqcap(Q_1, Q_2) & \sqcap^*(Q_1, Q_2)\\
\sqcap(Q_1, Q_3) & \sqcap^*(Q_1, Q_3)\\
\sqcap(Q_2, Q_3) & \sqcap^*(Q_2, Q_3)
\end{bmatrix} 
\in 
\left\{
\begin{bmatrix}
0 & 0\\
0 & 1\\
1 & 0
\end{bmatrix},
\begin{bmatrix}
0 & 0\\
0 & 0\\
1 & 0
\end{bmatrix},
\begin{bmatrix}
0 & 0\\
0 & 0\\
0 & 1
\end{bmatrix} 
\right\}.
$$
\end{itemize} 
These three types are called, respectively,  {\it mixed nasty},  {\it plane nasty}, and {\it dual-plane nasty}. 
 Clearly, $\QQ$ is  plane-nasty in $M$ if and only if $\QQ$ is dual-plane nasty in $M^*$; and 
 $\QQ$ is mixed nasty in $M$ if and only if $\QQ$ is mixed nasty in $M^*$.

%\begin{figure}[tbh]
%\includegraphics[scale=.35]{flxi.pdf} 
%\end{figure}

%Let $\QQ=(Q_L,Q_1,\ldots,Q_n,Q_R)$ be a $(4,2)$-flexipath in $M$. 
We say that $Q_i$ is a 
 {\em specially placed step} in a $(4,2)$-flexipath $(L,Q_1,Q_2,\ldots,Q_n,R)$ in $M$ if either 
\begin{itemize}
\item[(S1)] $\sqcap(L,R)=2$ and $\sqcap(L,Q_i) = 2 = \sqcap(R,Q_i)$; or %$Q_i\subseteq \cl(L) \cap  \cl(R)$
\item[(S2)] $\sqcap^*(L,R)=2$ and  $\sqcap^*(L,Q_i) = 2 = \sqcap^*(R,Q_i)$. % $Q_i\subseteq \cl^*(L) \cap \cl^*(R)$
\end{itemize}

Evidently, $Q_i$ is a specially placed step of type (S2) in $M$ if and only if   $Q_i$ is a specially placed step of type (S1) in $M^*$.
Specially placed steps are not particularly problematic for, as we now show, there is at most one of them. In this and the remaining results in this section, 
$(L,Q_1,Q_2,\ldots,Q_n,R)$ is a $(4,2)$-flexipath $\QQ$.

\begin{lemma}
\label{notmany}
$\QQ$ has at most one specially placed step.
\end{lemma}

\begin{proof}
Assume that $Q_1$ and $Q_2$ are both specially placed elements of type (S1). For the rest of the argument, we will  again be operating in the clonal core. There, since $\sqcap(L,Q_i) = 2$ for each $i$ in $\{1,2\}$, we deduce that 
$Q_1 \cup Q_2 \subseteq \cl(L)$. By symmetry, $Q_1 \cup Q_2 \subseteq \cl(R)$. Hence $Q_1 \cup Q_2 \subseteq \cl(L) \cap \cl(R)$. Thus 
$$\lambda(Q_1 \cup Q_2)  \le r(Q_1 \cup Q_2) \le r(\cl(L) \cap \cl(R)) \le \sqcap(L,R) = 2,$$
 a contradiction.

By duality, $\QQ$ has at most one specially placed step of type (S2). Now suppose that $Q_1$ is specially placed of type (S1), and 
$Q_2$ is specially placed of type (S2). Then $\sqcap(L,R) = 2 = \sqcap^*(L,R)$, so $\sqcap(L,R) + \sqcap^*(L,R)  = 4,$ 
a contradiction to Lemma~\ref{GG}. 
 %\begin{eqnarray*}
%3  \le \lm(L \cup R) & = &  \lm(L) + \lm(R) - \sqcap(L,R) - \sqcap^*(L,R)\\
%& =  & 3+3 - 2 -2 = 2, 
%\end{eqnarray*}
%a contradiction.
\end{proof}

\begin{lemma}
\label{LR3} 
If $\sqcap(L,R) = 3$, then 
$\sqcap(L,Q_i)  =  2$ for all $i$. 
\end{lemma}

\begin{proof} We may assume that $M$ is a clonal-core matroid. Then $r(L) = r(R) = 3$ and, since $\sqcap(L,R) = 3$, this implies that $r(L \cup R) = 3$. 
If $n = 1$, then $r(M) = 3$ and $\sqcap(L,Q_1) = r(Q_1) = 2$. 
Next assume that $n = 2$. We have 
$$3 \le \lm(Q_1 \cup Q_2) = r(Q_1 \cup Q_2) + r(L \cup R) - r(M).$$ 
But $r(L \cup R) = 3$ since $\sqcap(L,R) = 3$. Thus 
$$3 \le r(M) \le r(Q_1 \cup Q_2) \le 4.$$ Suppose $r(M) = 3$. Then $\sqcap(L,Q_i) = 2$ for all $i$. Thus we may assume that 
$r(M) = 4$. Then 
$$3 = \lm(L \cup Q_1) = r(L\cup Q_1) + r(R \cup Q_2) - r(M).$$
Thus $r(L\cup Q_1) + r(R \cup Q_2) = 7$. Hence we may assume that $r(L \cup Q_1) = 3$. Then $r(L \cup Q_1 \cup R) = 3$, so $\lm(Q_2) = 1$, a contradiction. 
We conclude that the result holds for $n = 2$. 

Assume the result holds for $n < k$ and let $n= k \ge 3$. Consider the path $(L,Q_1,Q_2, Q_3 \cup Q_4 \cup \dots \cup Q_k \cup R)$ of $4$-separations. By applying the result for $n = 2$ to the clonal core of 
$(M,\{L,Q_1,Q_2, Q_3 \cup Q_4 \cup \dots \cup Q_k \cup R\})$, we deduce that $\sqcap(L,Q_1) = 2$ and the lemma follows by induction.
\end{proof}

\begin{lemma}
\label{flexiguts} 
Let $n \ge 2$. Assume that $\QQ$ has no specially placed steps of type {\em (S1)}. 
If $\sqcap(L,Q_i) = 2$ for some $i$ in $[n]$, then $\sqcap(L,R) = 3$ and
$\sqcap(L,Q_j)  = 2 =  \sqcap(R,Q_j)$ for all $j$ in $[n]$. 
\end{lemma}

\begin{proof} Again we may assume that $M$ is a clonal-core matroid. 
We may also assume that $i = 1$. 
Suppose first that $n = 2$. Then $r(L \cup Q_1) = 3$ and, by Lemma~\ref{A}, $\sqcap(R,Q_1) = 2$. Thus $\sqcap(L,R) \ge 2$. If $\sqcap(L,R) = 2$, then $Q_1$ is a specially placed step of type (S1), a contradiction. Thus 
$\sqcap(L,R) = 3$. Hence $r(L \cup R) = 3 = r(L \cup R \cup Q_1),$ so $r(M) = 3$ and $\sqcap(L,Q_2)  =  \sqcap(R,Q_2) = 2$.  

We now know the result holds for $n = 2$. Assume it holds for $n<k$ and let $n = k \ge 3$. Then, by considering the path $(L,Q_1,Q_2, Q_3 \cup Q_4 \cup \dots \cup Q_k \cup R)$ of $4$-separations and applying the induction assumption to the clonal core of 
$(M,\{L,Q_1,Q_2, Q_3 \cup Q_4 \cup \dots \cup Q_k \cup R\})$, we deduce that $\sqcap(L,Q_2) = 2$. Because we are dealing with a $(4,2)$-flexipath, we get that $\sqcap(L,Q_j) = 2$ for all $j$ in $\{1,2,\dots,k\}$. Then 
$r(L \cup Q_1 \cup Q_2 \cup \dots \cup Q_k) = 3$. Thus $r(M) = 3$ and $\sqcap(L,R) = 3$. We conclude, by induction, that the lemma holds.
\end{proof}

%The next result is a straightforward consequence of the last lemma, so we omit the proof. 

\begin{lemma}
\label{get-squashed} 
Assume that the $(4,2)$-flexipath $\QQ$ has at least two internal steps and has no specially placed steps of type {\em (S1)}. 
If $\sqcap(L,Q_i) = 2$  for some $i$ in $[n]$, then $\QQ$ is a squashed $(4,2)$-flexipath.
\end{lemma}

\begin{proof}
By Lemma~\ref{flexiguts}, $\sqcap(L,R) = 3$ and
$\sqcap(L,Q_j)  = 2 =  \sqcap(R,Q_j)$ for all $j$ in $[n]$. 
By Lemmas~\ref{GG} and \ref{B}, $\sqcap^*(L,R) = 0$, and $\sqcap^*(L,Q_j)  = 0 =  \sqcap^*(R,Q_j)$ for all $j$ in $[n]$. 
Finally, working in the clonal core, we have $r(L) = 3 = r(L \cup Q_j)$ and $r(Q_j) = 2$ for all $j$. Thus $\sqcap(Q_g,Q_h) \ge 1$ for all distinct $g$ and $h$. 
Thus, by Lemma~\ref{C}, $\sqcap(Q_g,Q_h) = 1$ and $\sqcap^*(Q_g,Q_h) = 0$ for all distinct $g$ and $h$. Hence $\QQ$ is a squashed $(4,2)$-flexipath.
\end{proof}

The dual of the last lemma is the following. 

\begin{lemma}
\label{stretch-your-quads} 
Assume that the $(4,2)$-flexipath $\QQ$ has at least two internal steps and has no specially placed steps of type {\em (S2)}. 
If $\sqcap(L,Q_i) = 0$  for some $i$ in $[n]$, then $\QQ$ is a stretched $(4,2)$-flexipath.
\end{lemma}

\begin{lemma}
\label{F} 
Assume that the $(4,2)$-flexipath $\QQ$ has at least two internal steps and has no specially placed steps. 
If $\QQ$ is neither a squashed nor a stretched $(4,2)$-flexipath, then, for all $i$ in $[n]$, 
$$\sqcap(L,Q_i) = \sqcap^*(L,Q_i)  = 1 = \sqcap(R,Q_i) = \sqcap^*(R,Q_i).$$  
\end{lemma}

\begin{proof} 
By Lemmas~\ref{get-squashed} and \ref{stretch-your-quads}, $\sqcap(L,Q_i) = 1$, for all $i$ in $[n]$. Thus, by Lemma~\ref{B}, $\sqcap^*(L,Q_i) = 1$ for all $i$. 
Moreover, by Lemma~\ref{A}, $\sqcap(R,Q_i) = 1 = \sqcap^*(R,Q_i)$ for all $i$. 
\end{proof}

%Next we address the V\'{a}mos-like outcomes. 
Recall that, for a non-empty subset $J$ of $[n]$, we are abbreviating  $\cup_{j \in J} Q_j$ as 
$Q_J$. When $J$ is empty, so is $Q_J$.

\begin{lemma}
\label{getting-interesting}
Let $\QQ$  be a $(4,2)$-flexipath $(L,Q_1,Q_2, \ldots,Q_n,R)$ in  a matroid $M$ where $n \ge 2$. 
Assume that $\QQ$ is neither squashed nor stretched and has no specially placed steps.
Then 
\begin{itemize}
\item[(i)] for all $J\subseteq [n] - \{i\}$, 
$$\sqcap(L\cup Q_J, Q_i)=\sqcap(Q_i,Q_J\cup R)= 1 = 
\sqcap^*(L\cup Q_J, Q_i)=\sqcap^*(Q_i,Q_J\cup R);$$
\item[(ii)] $r(L \cup Q_J) = r(L) + \sum_{j \in J} r(Q_j)- |J|$ for all $J \subseteq [n]$;
\item[(iii)] $r(M)=r(L)+r(Q_1)+ r(Q_2) + \dots +r(Q_n) + r(R) -n-3$.
\end{itemize}
\end{lemma}

\begin{proof}
By Lemma~\ref{F}, for all $i$ in $[n]$, we have 
$$\sqcap(L, Q_i)=\sqcap(Q_i, R)= 1 = 
\sqcap^*(L, Q_i)=\sqcap^*(Q_i, R).$$ 
To prove (i), we may assume that $i = 1$ and $J= \{2,3,\dots,j\}$. 
Then $(L \cup Q_J, Q_1,Q_{j+1},\dots,Q_n,R)$ is a path of $4$-separations and 
$\sqcap(R,Q_1) = 1$, so, by Lemma~\ref{A}, $\sqcap(L \cup Q_J,Q_1) = 1$. Thus (i) holds.

By (i), we have $r(L \cup Q_1) = r(L) + r(Q_1) - 1$ and
$$r(L \cup Q_1 \cup Q_2 \cup \dots \cup Q_j) = r(L \cup Q_1 \cup Q_2 \cup \dots \cup Q_{j-1}) + r(Q_j) - 1,$$
so $r(L \cup Q_J) = r(L) + r(Q_1) + r(Q_2) + \dots + r(Q_j) - j,$  so (ii) holds. In particular, 
$$r(L \cup Q_1 \cup Q_2 \cup \dots \cup Q_n) = r(L) + r(Q_1) + r(Q_2) + \dots + r(Q_n) - n.$$
As $\sqcap(L \cup Q_1 \cup Q_2 \cup \dots \cup Q_n,R) = 3$, we deduce that 
$$r(M)=r(L)+r(Q_1)+ r(Q_2) + \dots +r(Q_n) + r(R) -n-3,$$
so (iii) holds.
\end{proof}

\begin{lemma}
\label{clonal_core}
In the clonal core of $M$, for distinct $i$ and $j$, 
\begin{itemize}
\item[(i)] $\sqcap(Q_i,Q_j) = 1$ if and only if $Q_i \cup Q_j$ is a circuit;
\item[(ii)] $\sqcap(Q_i,Q_j) = 0$ if and only if $Q_i \cup Q_j$ is independent;
\item[(iii)] $\sqcap^*(Q_i,Q_j) = 1$ if and only if $Q_i \cup Q_j$ is a cocircuit; and 
\item[(iv)] $\sqcap^*(Q_i,Q_j) = 0$ if and only if $Q_i \cup Q_j$ is coindependent.
\end{itemize}
\end{lemma}

\begin{proof}
By duality, it suffices to prove (i) and (ii). We have 
 \begin{eqnarray*}
\sqcap(Q_i,Q_j) & = &   r(Q_i) +r(Q_j) - r(Q_i \cup Q_j)\\
& = & 4 - r(Q_i \cup Q_j).
\end{eqnarray*} 
Thus $r(Q_i \cup Q_j) = 4 - \sqcap(Q_i,Q_j)$. Because the elements of $Q_j$ are clones, parts (i) and (ii) follow immediately.
\end{proof}

%% I think the next result is meant for the clonal core of a matroid and that probably need to be added to the hypothesis. The theorm that follows can be inferred from this 
%% by using the links beteeem $\lambda$ and $\sqcap$ in a matroid and its clonal core.

\begin{lemma}
\label{LR0} 
Let $\QQ$ be a $(4,2)$-flexipath $(L,Q_1,Q_2,\ldots,Q_n,R)$ in a  matroid. Assume that $\QQ$ is neither squashed nor stretched and has no specially placed steps.
If $\sqcap(L,R) = 0$, then, in the clonal core $M$ of the matroid, 
\begin{itemize}
\item[(i)] $\sqcap(L,Q_i) = \sqcap^*(L,Q_i) = 1 = \sqcap(R,Q_i) = \sqcap^*(R,Q_i)$ for all $i$; 
\item[(ii)] $n=3$;
\item[(iii)] $r(M) = 6 = r^*(M)$; 
\item[(iv)]  $\sqcap^*(Q_i,Q_j) = 0$ for all distinct $i$ and $j$;
\item[(v)] $\sqcap^*(L,R) = 6 - r(Q_1 \cup Q_2 \cup Q_3)$;
\item[(vi)] $r(Q_1 \cup Q_2 \cup Q_3) \in \{4,5,6\}$; 
\item[(vii)] if $r(Q_1 \cup Q_2 \cup Q_3) = 6$, then $\QQ$ is prism-like; 
\item[(viii)] if $r(Q_1 \cup Q_2 \cup Q_3) = 4$, then $\sqcap(Q_i,Q_j) = 1$ for all distinct $i$ and $j$, and 
$\QQ$ is relaxed-spike-reminiscent; and 
\item[(ix)] if $r(Q_1 \cup Q_2 \cup Q_3) = 5$, then either 
\begin{itemize}
\item[(a)] $\sqcap(Q_i,Q_j)=0$ for all distinct $i$ and $j$ in $\{1,2,3\}$, and $\QQ$ is tightened-prism-like; or 
\item[(b)] after a possible permutation of $\{1,2,3\}$, 
$$\sqcap(Q_1,Q_2)=0 = \sqcap(Q_1,Q_3) \text{~~and~~} \sqcap(Q_2,Q_3) = 1,$$
and $\QQ$ is V\'{a}mos-inspired.
\end{itemize}
\end{itemize}
\end{lemma}

\begin{proof}
%We assume that we are dealing with the clonal core of $M$. 
Suppose $n = 1$. Then, since we may assume that $M$ is a clonal-core matroid, 
$3=\lambda(L\cup Q_1)=r(L\cup Q_1)+r(R)-r(M)$. Since $r(R)=3$, we have $r(L\cup Q_1)=r(M)$. But $\sqcap(L, R)=0$, we see that $r(M)\ge 6$, while $r(L\cup Q_1)\le 5$, a contradiction. Hence $n \ge 2$. 
Part (i) is immediate from Lemma~\ref{F}. 
Then, as
$\sqcap(L,Q_i) = 1 = \sqcap(R,Q_i)$, Lemma~\ref{jamess} gives that  
\begin{eqnarray*}
 2 \ge \sqcap(Q_i, L \cup R) + \sqcap(L,R) &=&  \sqcap(Q_i\cup L, R) + \sqcap(Q_i, L)\\
 & \ge& \sqcap(Q_i, R) + \sqcap(Q_i, L) = 2.
 \end{eqnarray*}
 Thus $\sqcap(Q_i,L\cup R) = 2.$ Hence $L \cup R$ spans $M$ so 
 $r(M) = r(L \cup R) = 6$. By Lemma~\ref{getting-interesting}(iii), 
 $6 = r(L) + 2n - n$, so $n = 3$ and (ii) holds. Moreover, for each distinct $i$ and $j$, we see that $Q_i \cup Q_j$ is coindependent. Thus, by 
 Lemma~\ref{clonal_core}(iv), $\sqcap^*(Q_i,Q_j) = 0$, that is, (iv) holds.
 
 Since  $r(M) = 6$ and 
 $$|E(M)| = 3+ 2+ 2+2 +3 = 12,$$ 
 we see that $r^*(M) = 6$, that is, (iii) holds. Now, by Lemma~\ref{geoffs},  
 \begin{eqnarray}
 \label{eq*}
\sqcap^*(L,R) &=&  \lm(L) + \lm(R) - \sqcap(L,R) - \lm(L \cup R) \nonumber \\
& = & 3 + 3 - 0 - (r(L \cup R) + r(Q_1 \cup Q_2 \cup Q_3) - r(M)) \nonumber \\
&=& r(M) - r(Q_1 \cup Q_2 \cup Q_3)\nonumber \\
 &=&  6  - r(Q_1 \cup Q_2 \cup Q_3).
 \end{eqnarray}
 Thus (v) holds.
 
 As $\lm(Q_1 \cup Q_2) \ge 3$, we see that $r(Q_1\cup Q_2) \ge 3$. Suppose 
 $r(Q_1\cup Q_2 \cup Q_3) = 3$. Then 
 $$3 = \lm(R) = r(R) + r(L\cup Q_1 \cup Q_2 \cup Q_3) - r(M).$$ Thus 
 $r(L\cup Q_1 \cup Q_2 \cup Q_3) =  r(M) = 6$. As 
 $r(L) = 3 = r(Q_1 \cup Q_2 \cup Q_3)$, it follows that 
 $\sqcap(L,Q_1 \cup Q_2 \cup Q_3) = 0$, so $\sqcap(L,Q_1) = 0$, a contradiction. Hence $r(Q_1 \cup Q_2 \cup Q_3) \ge 4,$ so (vi) holds. 
 
 Now suppose that  $r(Q_1 \cup Q_2 \cup Q_3)   = 6$. Then $\sqcap(Q_i,Q_j) = 0$ for all distinct $i$ and $j$, and $\sqcap^*(L,R) = 0$ by (v). 
 It follows %by Lemma~\ref{B} 
 that $\QQ$ is prism-like, so (vii) holds. 
 
Next suppose that  $r(Q_1 \cup Q_2 \cup Q_3)   = 4$. Then $\sqcap^*(L,R) = 2$, so $r^*(L\cup R) = 4$. We now show that
\begin{sublemma}
\label{4.1}
$\sqcap(Q_i,Q_j) = 1$  for all distinct $i$ and $j$.
\end{sublemma}

We have that 
 \begin{eqnarray*}
r(Q_1 \cup Q_2) &=&  |Q_1 \cup Q_2| + r^*(L \cup Q_3 \cup R) - r^*(M)\\
& = & r^*(L \cup Q_3 \cup R) - 2. 
 \end{eqnarray*}
 Now $$2 = r^*(Q_3)  \ge \sqcap^*(Q_3,L\cup R) \ge \sqcap^*(Q_3,L) = 1$$
 where the last inequality follows by monotonicity.  
 If $\sqcap^*(Q_3, L \cup R) = 2$, then 
 $r^*(L \cup Q_3 \cup R) = r^*(L \cup R) = 4$, 
 so 
\begin{eqnarray*}
\lambda(Q_1 \cup Q_2) &  = & r^*(Q_1 \cup Q_2) + r^*(L \cup R\cup Q_3) - r^*(M)\\
& \le & 4 + 4 - 6 = 2,
\end{eqnarray*}
a contradiction. Thus $\sqcap^*(Q_3, L \cup R) = 1$, so $r^*(L \cup Q_3 \cup R) = 5$ and $\sqcap(Q_1,Q_2) = 1$. 
We conclude, by symmetry, that \ref{4.1} holds, so (viii) holds.

Finally suppose that    $r(Q_1 \cup Q_2 \cup Q_3)   = 5$. Then, by (\ref{eq*}), $\sqcap^*(L,R) = 1$, so $r^*(L \cup R) = 5$. 
If $r^*(L \cup R \cup Q_1) = 5$ and $r^*(L \cup R \cup Q_2) = 5$, then $r^*(L \cup R \cup Q_1 \cup Q_2) = 5$. But this gives a contradiction as $r^*(Q_3) = 2$ and $\lm(Q_3) = 2$. 
Thus, by potentially taking a permutation of $\{1,2,3\}$, we may assume that 

\begin{itemize}
\item[(a)] $r^*(L \cup R \cup Q_i) = 6$ for all $i$; or  
\item[(b)] $r^*(L \cup R \cup Q_1) = 5$ and $r^*(L \cup R \cup Q_2) = 6 = r^*(L \cup R \cup Q_3)$.
\end{itemize}

In case (a), we have, using the formula for the rank function in the dual of a matroid, 
 \begin{eqnarray*}
6 &=&  r^*(L \cup R \cup Q_3)\\
& = & |L \cup R \cup Q_3| + r(Q_1 \cup Q_2) - r(M)\\
& = & 8 +r(Q_1 \cup Q_2) - 6.
\end{eqnarray*}
Hence $r(Q_1 \cup Q_2) = 4$. Similarly, $\sqcap(Q_i,Q_j) = 0$ for all distinct $i$ and $j$. Thus $\QQ$ is tightened-prism-like. 

In case (b), $\sqcap(Q_1,Q_2) = 0 = \sqcap(Q_1,Q_3)$ and $\sqcap(Q_2,Q_3) = 1$. Thus $\QQ$ is V\'{a}mos-inspired. 
\end{proof}

Following Lemma~\ref{1,1,3cor}, we  provide specific examples of matroids that satisfy (viii), (ix)(a), and (ix)(b) of the last lemma.

\begin{lemma}
\label{bign} 
Let $(L,Q_1,Q_2,\dots,Q_n,R)$ be a $(4,2)$-flexipath $\QQ$ with no specially placed steps. Assume that $\QQ$ is neither squashed nor stretched. Suppose $n \ge 2$ and $n \neq 3$. If $\sqcap(L,R) = 2$, then 
 \begin{itemize}
\item[(i)] $\sqcap^*(L,R) = 1$;
\item[(ii)] for all $i$ in $[n]$ and all $J \subseteq [n] - \{i\}$,
$$\sqcap(L,Q_i)  = \sqcap(L \cup Q_J,Q_i) = 1 = \sqcap^*(L,Q_i)  = \sqcap^*(L \cup Q_J,Q_i);$$
%\item[(iii)] $\sqcap(Q_i,Q_j) = 0$ for all distinct $i$ and $j$ in $[n]$;
\item[(iii)] $r(L \cup Q_J) = r(L) + \sum_{j \in J} r(Q_j)- |J|$ for all $J \subseteq [n]$; 
\item[(iv)] $r(M) = r(L) + \sum_{i = 1} ^ n r(Q_i) + r(R) - n - 3;$ 
\item[(v)]  $r(Q_i \cup Q_j) = r(Q_i) + r(Q_j)$,   for all distinct $i$ and $j$ in $[n]$; 
\item[(vi)] $r(Q_J) =  \sum_{j \in J} r(Q_j)- |J| +2$ for all $J \subseteq [n]$ such that $|J| \ge 2$; and  
\item[(vii)] $r(L \cup R \cup Q_J) = r(L) + r(R) +  \sum_{j \in J} r(Q_j)- |J| -2$ for all $J \subseteq [n]$ such that $ 2 \le |J| \le n-1$.
\end{itemize}
\end{lemma}

\begin{proof}
By Lemma~\ref{GG}, $\sqcap(L,R) + \sqcap^*(L,R) \le 3$. As $\sqcap(L,R) = 2$, we deduce that $\sqcap^*(L,R) \le 1$. If $\sqcap^*(L,R) = 0$, then, by Lemma~\ref{LR0}, and duality, 
$n = 3$, a contradiction. Thus $\sqcap^*(L,R) = 1$, so (i) holds. Parts (ii), (iii), and (iv) repeat parts (i), (ii), and (iii) of Lemma~\ref{getting-interesting}.

For (v) and (vi), since $\sqcap(L,R)  = 2$, we have $r(L\cup R) = r(L) + r(R) - 2$. As 
$r(L\cup Q_3 \cup Q_4 \cup \dots \cup Q_n) = r(L) + r(Q_3) + r(Q_4) + \dots + r(Q_n) - (n-2)$, we see that 
 \begin{eqnarray*}
r(L\cup R \cup Q_3 \cup Q_4 \cup \dots \cup Q_n)&=& r(L) + r(R)  + \sum_{i = 3}^{n} r(Q_i) - (n-2)\\ 
& & - \sqcap(R, L \cup Q_3 \cup Q_4 \cup \dots \cup Q_n)\\
& \le  & r(L) + r(R) +  \sum_{i = 3}^{n} r(Q_i)  - n,
\end{eqnarray*} 
where the last step follows because $$\sqcap(R, L\cup Q_3 \cup Q_4 \cup \dots \cup Q_n) \ge \sqcap(R,L) = 2.$$ Thus
 \begin{eqnarray*}
 3 \le \lm(Q_1 \cup Q_2) & = & r(Q_1 \cup Q_2) + r(L\cup R \cup Q_3 \cup Q_4 \cup \dots \cup Q_n) - r(M)\\
 & \le & r(Q_1 \cup Q_2) + \sum_{i = 3}^{n} r(Q_i) - n + r(L) +r(R) \\
 && - \sum_{i = 1}^{n} r(Q_i)+ n+3 - r(L) - r(R).
\end{eqnarray*} 
Hence $r(Q_1) + r(Q_2) \le r(Q_1 \cup Q_2)$ so $\sqcap(Q_1,Q_2) = 0$. Thus (v) holds, so (vi) holds for $|J| = 2$. 

Now 
\begin{eqnarray*}
\sqcap(L,Q_1 \cup Q_2) &=& r(L)  + r(Q_1 \cup Q_2)   - r(L \cup Q_1 \cup Q_2) \\
&= & r(L)   +  r(Q_1) + r(Q_2)  - r(L)   -  r(Q_1)  -  r(Q_2) + 2\\
& =&  2.
\end{eqnarray*} 
Thus, for all subsets $J$ of $[n]$ with $|J| \ge 2$.
\begin{equation}
\label{2lq}
2 \le \sqcap(L,Q_J).
\end{equation} 

Since $\sqcap$ is monotonic, for a proper subset $J$ of $[n]$,
$$3 = \sqcap(L \cup Q_{[n] - J}, Q_J \cup R) \ge \sqcap(L, Q_J \cup R) \ge \sqcap(L,R) = 2.$$
If  $\sqcap(L, Q_J \cup R) = 3$, then, by Lemma~\ref{LR3}, $\sqcap(L,Q_i) = 2$ for all $i$ in $[n] - J$. But 
$\sqcap(L,Q_j) = 1$ for all $j$ in $[n]$, a contradiction. Hence $\sqcap(L, Q_J \cup R) = 2$ for all proper subsets $J$ of $[n]$.
 Combining this with (\ref{2lq}), we get that 
\begin{equation}
\label{a*1}
2 \le \sqcap(L,Q_J) \le \sqcap(L, Q_J \cup R) = 2 
\end{equation}
provided $2 \le |J| \le n-1$. Thus, for such $J$,
 \begin{eqnarray*}
 %\label{a0}
r(Q_J) &=& r(L \cup Q_J) - r(L) + \sqcap(L,Q_J) \\
& = & r(L) + \sum_{j \in J} r(Q_j) - |J| - r(L) + 2.
\end{eqnarray*} 
We have 
$$r(Q_{[n] - \{1\}}) + r(Q_{[n] - \{n\}}) \ge r(Q_{[n]}) + r(Q_{[n] - \{1,n\}}),$$ 
so
 \begin{eqnarray*}
\sum_{i = 2}^{n} r(Q_i) - n +  3
+  \sum_{i = 1}^{n-1} r(Q_i) - n + 3  
&\ge& r(Q_{[n]}) + \sum_{i = 2}^{n-1} r(Q_i)\\
&& - (n-2) + 2. 
\end{eqnarray*} 
Hence 
\begin{equation}\label{a1}
r(Q_{[n]}) \le \sum_{i = 1}^{n} r(Q_i) - n + 2.
\end{equation}
Also, as $\sqcap(L,R) +  \sqcap^*(L,R) \le 3$, it follows by Lemma~\ref{geoffs} that 
 \begin{eqnarray*}
3 & \le & \lm(L \cup R)\\
& = & r(Q_{[n]})  +r(L) + r(R) - 2 - r(M)\\
& = & r(Q_{[n]})  +r(L) + r(R) - 2 - r(L)  \\
&&- \sum_{i = 1}^{n} r(Q_i) - r(R) + n + 3.
\end{eqnarray*} 
Thus 
\begin{equation}
\label{a2}
\sum_{i = 1}^{n} r(Q_i) - n +2 \le r(Q_{[n]}).
\end{equation}
 Combining (\ref{a1}) and (\ref{a2}), we get 
 $$r(Q_{[n]}) = \sum_{i = 1}^{n} r(Q_i) - n +2.$$
 Hence, for all $J \subseteq [n]$ such that $|J| \ge 2$, we have 
 $$r(Q_{J}) = \sum_{j \in J} r(Q_j) -  |J| +2,$$
 that is, (vi) holds. 
 
 By (\ref{a*1}), $\sqcap(L, Q_J \cup R) = 2$ for all $J$ with $2 \le |J| \le n-1$, we have 
  \begin{eqnarray*}
r(L \cup Q_J \cup R) & = & r(L) + r(Q_J \cup R) - 2\\
& = & r(L) + r(R) + \sum_{j \in J} r(Q_j) -  |J| - 2.
\end{eqnarray*} 
We conclude that (vii) holds.
\end{proof}

Next, having dealt with the case when $\sqcap(L,R) = 0$ in Lemma~\ref{LR0}, we consider the case when $\sqcap(L,R) = 1$.

\begin{lemma}
\label{1,1,3} 
Let $\QQ$ be a $(4,2)$-flexipath $(L,Q_1,Q_2,Q_3,R)$ for which $\sqcap(L,R) = 1$. Then, in the clonal core of $M$,
$$r(M) = 6 = r^*(M),$$
and the following hold. 
\begin{itemize}
\item[(i)] If $\sqcap^*(L,R) = 1$, then $r(Q_1 \cup Q_2 \cup Q_3) = 5$ and, after a possible permutation of $\{1,2,3\}$, 
$$
\begin{bmatrix}
r(Q_1 \cup Q_2) & r^*(Q_1 \cup Q_2)\\
r(Q_1 \cup Q_3) & r^*(Q_1 \cup Q_3)\\
r(Q_2 \cup Q_3) & r^*(Q_2 \cup Q_3)
\end{bmatrix} 
\in 
\left\{
\begin{bmatrix}
4 & 4\\
4 & 3\\
3 & 4
\end{bmatrix},
\begin{bmatrix}
4 & 4\\
4 & 4\\
3 & 4
\end{bmatrix}, 
\begin{bmatrix}
4 & 4\\
4 & 4\\
4 & 3
\end{bmatrix},
\begin{bmatrix}
4 & 4\\
4 & 4\\
4 & 4
\end{bmatrix}
\right\}.
$$
\item[(ii)] If $\sqcap^*(L,R) = 2$, then $r(Q_1 \cup Q_2 \cup Q_3) = 4$. Moreover, for all distinct $i$ and $j$,
$$\sqcap(Q_i,Q_j) = 1  \text{~~ and ~~}  \sqcap^*(Q_i,Q_j) = 0,$$
so $r(Q_i \cup Q_j) = 3$ and $r^*(Q_i \cup Q_j) = 4$. In particular, $\QQ$ is spike-reminiscent.
\end{itemize}
\end{lemma}

\begin{proof}
We may assume that $M$ is a clonal-core matroid.  Thus $r(L) = 3 = r(R)$ and $r(Q_i) = 2$ for all $i$. 
By Lemma~\ref{getting-interesting}(ii),  $r(L \cup Q_1 \cup Q_2) = 5$ and $r(R \cup Q_3) = 4$. 
Thus 
$$3 = \lm(L \cup Q_1 \cup Q_2) = r(L \cup Q_1 \cup Q_2) + r(R \cup Q_3) - r(M).$$
Thus  
$r(M) = 6$, so $r^*(M) = 6$.

Next observe that, from the formula for the rank in the dual, we have

\begin{eqnarray}
\label{aabb}
r(Q_1 \cup Q_2 \cup Q_3) & = &  |Q_1 \cup Q_2 \cup Q_3| + r^*(L \cup R) - r^*(M) \nonumber \\
& = & 6 + (r^*(L) + r^*(R) - \sqcap^*(L,R)) - r^*(M) \nonumber \\
& = & 6 - \sqcap^*(L,R).
\end{eqnarray} 
%Thus
%\begin{equation} 
%\label{aabbc}
%\sqcap^*(L,R) = 6 - r(Q_1 \cup Q_2 \cup Q_3).
%\end{equation}

When $\sqcap^*(L,R) = 0$, by duality, we can deduce the structure of $M$ from Lemma~\ref{LR0}. Thus, we may assume  that $\sqcap^*(L,R) \ge 1$. 
By Lemma~\ref{GG}, $\sqcap^*(L,R) \le 2$.
Hence 
$r(Q_1 \cup Q_2 \cup Q_3) \ge 4$. 

Next we show that 
\begin{sublemma}
\label{5.7}
$r(Q_i \cup Q_j) + r^*(Q_i \cup Q_j) \ge 7$ for all distinct $i$ and $j$.
\end{sublemma}

This follows immediately since 
$$3 \le \lm(Q_i \cup Q_j) = r(Q_i \cup Q_j) + r^*(Q_i \cup Q_j) - |Q_i \cup Q_j|.$$

\begin{sublemma}
\label{5.3}
If $\sqcap^*(L,R) = 1$, then at most one of $r(Q_1 \cup Q_2)$, $r(Q_1 \cup Q_3)$, and $r(Q_2 \cup Q_3)$ is $3$.
\end{sublemma}

To see this, let $\{i,j,k\} = \{1,2,3\}$. Then, by \ref{aabb}, $r(Q_1 \cup Q_2 \cup Q_3) = 5$, so 
$$r(Q_i \cup Q_j) + (Q_i \cup Q_k) \ge r(Q_1 \cup Q_2 \cup Q_3) + r(Q_i) = 5 + 2 = 7.$$ 
Thus (\ref{5.3}) holds. 

By symmetry,
$$
\begin{bmatrix}
r(Q_1 \cup Q_2) & r^*(Q_1 \cup Q_2)\\
r(Q_1 \cup Q_3) & r^*(Q_1 \cup Q_3)\\
r(Q_2 \cup Q_3) & r^*(Q_2 \cup Q_3)
\end{bmatrix} 
\in 
\left\{
\begin{bmatrix}
4 & 4\\
4 & 3\\
3 & 4
\end{bmatrix},
\begin{bmatrix}
4 & 4\\
4 & 4\\
3 & 4
\end{bmatrix}, 
\begin{bmatrix}
4 & 4\\
4 & 4\\
4 & 3
\end{bmatrix},
\begin{bmatrix}
4 & 4\\
4 & 4\\
4 & 4
\end{bmatrix}
\right\}.
$$
Thus (i) holds.

\begin{sublemma}
\label{4.3}
If $\sqcap^*(L,R) = 2$, then   $r(Q_i \cup Q_j) = 3$  and $r^*(Q_i \cup Q_j) = 4$ for all distinct $i$ and $j$. 
\end{sublemma}

To see this, observe that $r^*(L \cup R) = 4$ as $\sqcap^*(L,R) = 2$. Now, using the formula for the rank function of the dual, we have 
 \begin{eqnarray*}
r(Q_1 \cup Q_2) & = &   |Q_1 \cup Q_2| + r^*(L \cup R \cup Q_3) - r^*(M)\\
& = &  r^*(L \cup R \cup Q_3) - 2.
\end{eqnarray*} 
Since $r(Q_1 \cup Q_2) \ge 3$, we deduce that $r^*(L \cup R \cup Q_3) \ge 5$. But, by Lemma~\ref{getting-interesting}(i),  
$$\sqcap^*(Q_3,L \cup R) \ge \sqcap^*(Q_3,L) = 1.$$
Thus $r^*(L \cup R \cup Q_3) \le 5$ so $r^*(L \cup R \cup Q_3) = 5$ and $r(Q_1 \cup Q_2) = 3$. It follows by symmetry that $r(Q_i \cup Q_j) = 3$ for all distinct $i$ and $j$. 
By \ref{5.7}, we deduce that $r^*(Q_i \cup Q_j) = 4$ for all distinct $i$ and $j$. Thus, by Lemma~\ref{getting-interesting}(i), $\QQ$ is spike-reminiscent, so  (ii) holds.
\end{proof} 

Combining Theorem~\ref{G4} with  Lemmas~\ref{getting-interesting}(i) and \ref{1,1,3} gives the following.

\begin{lemma}
\label{1,1,3cor} 
Let $\QQ$ be a $(4,2)$-flexipath $(L,Q_1,Q_2,Q_3,R)$ in a matroid $M$.
\begin{itemize} 
\item[(i)] If $\sqcap(L,R) = 1 = \sqcap^*(L,R)$, 
then, after a possible permutation of $\{1,2,3\}$, 
$$
\begin{bmatrix}
\sqcap(Q_1, Q_2) & \sqcap^*(Q_1, Q_2)\\
\sqcap(Q_1, Q_3) & \sqcap^*(Q_1, Q_3)\\
\sqcap(Q_2, Q_3) & \sqcap^*(Q_2, Q_3)
\end{bmatrix} 
\in 
\left\{
\begin{bmatrix}
0 & 0\\
0 & 1\\
1 & 0
\end{bmatrix},
\begin{bmatrix}
0 & 0\\
0 & 0\\
1 & 0
\end{bmatrix}, 
\begin{bmatrix}
0 & 0\\
0 & 0\\
0 & 1
\end{bmatrix}, 
\begin{bmatrix}
0 & 0\\
0 & 0\\
0 & 0
\end{bmatrix}
\right\}.
$$
In particular, in $M$ or $M^*$, the flexipath $\QQ$ is nasty or is doubly-tightened-prism-like. 
\item[(ii)] If $\sqcap(L,R) = 1$ and  $\sqcap^*(L,R) = 2$, then $\QQ$ is spike-reminiscent. 
\item[(iii)] If $\sqcap(L,R) = 2$ and  $\sqcap^*(L,R) = 1$, then $\QQ$ is paddle-reminiscent. 
\end{itemize}
\end{lemma}

Next we provide examples of matroids satisfying (viii), (ix)(a) and (ix)(b) of Lemma~\ref{LR0}. We also provide examples of a doubly-tightened-prism-like $(4,2)$-flexipath and of one of the  
types of nasty $(4,2)$-flexipaths. To explain this, we consider the operation of tightening a basis. Following Ferroni and Vecchi~\cite{fv}, we call a basis $B$ in a matroid $M$ a {\it free basis} if $0 < r(M) < |E(M)|$ and 
$B \cup \{e\}$ is a circuit for all $e$ in $E(M) - B$. Equivalently, $B$ is a free basis of $M$ if it is not the unique basis of $M$ and 
 every fundamental circuit with respect to $B$ is spanning. As is well known (see, for example, \cite[Exercise 1.5.14]{oxbook}), a matroid $M$ is a relaxation of another matroid $N$ if and only if $M$ has a free basis $B$, in which case, $B$ is a circuit-hyperplane of $N$. We call $N$ a {\it tightening} of $M$. Formally, $E(N) = E(M)$ and ${\mathcal B}(N) = {\mathcal B}(M) - \{B\}$. 

For a matroid satisfying Lemma~\ref{LR0}(viii), begin with a rank-$7$ free spike whose legs are $\{x_i,y_i\}$ for all $i$ in $\{1,2,\dots,7\}$. Add elements $\alpha_1$, $\alpha_2$, and $\alpha_3$ freely to the plane spanned by $\{x_1,y_1,x_2,y_2\}$. 
Then add elements $\beta_1$, $\beta_2$, and $\beta_3$ freely to the plane spanned by $\{x_6,y_6,x_7,y_7\}$.  Let $Q_i = \{x_{i+2},y_{i+2}\}$ for each $i$ in $\{1,2,3\}$. Now truncate this matroid to rank 6, and delete 
$\{x_1,y_1,x_2,y_2,x_6,y_6,x_7,y_7\}$. 
Let $L = \{\alpha_1, \alpha_2, \alpha_3\}$
and $R= \{\beta_1, \beta_2, \beta_3\}$. In the matroid $M$ that we now have, $L \cup R$ is a circuit-hyperplane. Moreover, $(L,Q_1,Q_2,Q_3,R)$ is spike-reminiscent in $M$. 
In $M$, relax the circuit-hyperplane  $L \cup R$  to get  a rank-6 matroid $M_8$ with a $(4,2)$-flexipath $(L,Q_1,Q_2,Q_3,R)$ in which $\sqcap(L,R) = 0$ and $r(Q_1 \cup Q_2 \cup Q_3) = 4$. It is not difficult to check that 
$M_8$ satisfies Lemma~\ref{LR0}(viii). Indeed, $(L,Q_1,Q_2,Q_3,R)$ is relaxed-spike-reminiscent in $M_8$.

To give examples of a tightened-prism-like and doubly-tightened-prism-like flexipaths, we begin by giving an example of a prism-like matroid. Begin with a 6-element independent set $\{b_1,b_2,\dots,b_6\}$. Now, for each $i$ in $\{1,2,3\}$, freely add two points, $x_i$ and $y_i$, on the line spanned by 
$\{b_i,b_{i+3}\}$, and let $Q_i = \{x_i,y_i\}$. Now freely add points $\alpha_1$, $\alpha_2$, and $\alpha_3$ to the plane spanned by $\{b_1,b_2,b_3\}$. Similarly, freely add points $\beta_1$, $\beta_2$, and $\beta_3$  to the plane spanned by 
$\{b_4,b_5,b_6\}$. Now delete $\{b_1,b_2,\dots,b_6\}$, and let $L = \{\alpha_1, \alpha_2, \alpha_3\}$ and $R= \{\beta_1, \beta_2, \beta_3\}$. In this  rank-6 matroid $M$, we have   a $(4,2)$-flexipath $(L,Q_1,Q_2,Q_3,R)$ that is prism-like. 
Moreover, in $M$, the set $\{x_1,y_1,x_2,y_2,x_3,y_3\}$ is a free basis $B$. Let $N$ be the matroid that is obtained by tightening $B$.  In $N$, one can easily  check that $(L,Q_1,Q_2,Q_3,R)$ is a tightened-prism-like flexipath, that is,  Lemma~\ref{LR0}(ix)(a) holds. In $N$, we see that $\{\alpha_1, \alpha_2, \alpha_3,\beta_1, \beta_2, \beta_3\}$ is a free basis $B_N$. Let $P$ be the matroid obtained from $N$ by tightening $B_N$. In that case, $(L,Q_1,Q_2,Q_3,R)$ is a doubly-tightened-prism-like flexipath in $P$.

To describe a matroid satisfying Lemma~\ref{LR0}(ix)(b), begin with a V\'{a}mos matroid $V$ with ground set $\{a_1,a_2,b_1,b_2,c_1,c_2,d_1,d_2\}$ where $\{a_1,a_2,d_1,d_2\}$ is a basis and the only non-spanning circuits are the circuit-hyperplanes $\{a_1,a_2,b_1,b_2\}$, 
$\{a_1,a_2,c_1,c_2\}$, $\{b_1,b_2,c_1,c_2\}$, $\{b_1,b_2,d_1,d_2\}$, and $\{c_1,c_2,d_1,d_2\}$. Let $A = \{a_1,a_2\}$ and $D = \{d_1,d_2\}$. Take the direct sum of $V$ and $U_{2,2}$ where the latter has ground set $\{a,d\}$. 
Now freely add points  $\alpha_1$, $\alpha_2$, $\alpha_3$ and $\alpha_4$ to the plane spanned by $A \cup \{a\}$. Similarly, freely add points  $\delta_1$, $\delta_2$, $\delta_3$ and $\delta_4$ to the plane spanned by $D \cup \{d\}$. On the line spanned by 
$\alpha_4$ and $\delta_4$, freely add points $\beta_1$ and $\gamma_1$. Let $Q_1 = \{\beta_1,\gamma_1\}$. Delete $\{a_1,a_2,d_1,d_2,a,d,\alpha_4,\delta_4\}$ to give a matroid $M_9$. 
Let $Q_2 = \{b_1,b_2\}$ and $Q_3  = \{c_1,c_2\}$. Then $\sqcap(Q_1,Q_2) =0 = \sqcap(Q_1,Q_3)$ and $\sqcap(Q_2,Q_3) =1$, while $\sqcap^*(Q_i,Q_j) = 0$ for all distinct $i$ and $j$. Let $L = \{\alpha_1, \alpha_2, \alpha_3\}$ and $R= \{\delta_1, \delta_2, \delta_3\}$. In $M_9$, we now have that $(L,Q_1,Q_2,Q_3,R)$ 
is a $(4,2)$-flexipath that satisfies Lemma~\ref{LR0}(ix)(b), that is, $(L,Q_1,Q_2,Q_3,R)$  is V\'{a}mos-inspired. 

We can modify the last example to get an example of one of the  types of nasty $(4,2)$-flexipaths. In the matroid $M_9$, we have that $\sqcap(L,R) = 0$ and $\sqcap^*(L,R) = 1$. In this matroid, we see that $L \cup R$ is a free basis. Tightening this basis gives a matroid $N_9$ in which $\sqcap(L,R) = 1$ and $\sqcap^*(L,R) = 1$. Moreover, in $N_9$, we have that $\sqcap(Q_1,Q_2) = 0 = \sqcap(Q_1,Q_3)$ and $\sqcap(Q_2,Q_3) =1$, while $\sqcap^*(Q_i,Q_j) = 0$ for all distinct $i$ and $j$, and 
$r(Q_1 \cup Q_2 \cup Q_3) = 5$. Thus, in $N_9$, we see that $(L,Q_1,Q_2,Q_3,R)$ is an example of the second type of nasty $(4,2)$-flexipath. By dualizing, we get an example of the third  type of nasty $(4,2)$-flexipath. To give an example  of the first type of nasty $(4,2)$-flexipath, we again use the technique described at the end of Section~\ref{beh}.   Because we can determine the ranks of all subsets of $\{L,R,Q_1,Q_2,Q_3\}$, we can   check that if $X$ and $Y$ are such subsets, then  $r(X) + r(Y) \ge r(X \cup Y) + r(X \cap Y)$. Thus we have a polymatroid on the set of subsets of $\{L,Q_1,Q_2,Q_3,R\}$. By freely adding three elements to each of $L$ and $R$ and freely adding two points to each of $Q_1$, $Q_2$, and $Q_3$, we get a new polymatroid. Restricting this polymatroid to the the set of twelve newly added elements, we get a matroid that exemplifies the  first type of nasty $(4,2)$-flexipath.

\begin{lemma}
\label{1.14} 
Let $\QQ$ be a $(4,2)$-flexipath $(L,Q_1,Q_2,\dots,Q_n,R)$ where $\sqcap(L,R) = 1 = \sqcap^*(L,R)$,  and $\sqcap(L,Q_i) = 1$ for all $i$ in $[n]$. 
Then $\QQ$ has at most three internal steps.
\end{lemma}

\begin{proof} 
We assume that $n \ge 4$. We may assume that $M$ is a clonal-core matroid.  Then 
$r(L) = 3$ and $r(L \cup Q_i) = 4$ for all $i$. Moreover, by Lemma~\ref{getting-interesting}(i), for all distinct $i, j$, and $k$, we have 
$r(L \cup Q_i \cup Q_j) = 5$ and $r(L \cup Q_i \cup Q_j \cup Q_k)  = 6$. Also 
\begin{eqnarray*}
\sqcap(L , R \cup Q_i) & = &   r(L) + r(R \cup Q_i) - r(L \cup R \cup Q_i) \\
& \le &  r(L) + r(R \cup Q_i) - r(L \cup R)\\
& = &  3 + 4  - 5 = 2.
\end{eqnarray*}

Next we show the following. 
\begin{sublemma}
\label{sub-2}
For distinct $i$ and $j$ in $[n]$, if
  $\sqcap(L, R \cup Q_i) =  2 = \sqcap(L, R \cup Q_j)$, then $r(L \cup R \cup Q_i \cup Q_j) \le 5$.
\end{sublemma}

We have 
\begin{eqnarray*}
r(L \cup R \cup Q_i \cup Q_j) & \le &   r(L \cup R \cup Q_i) + r(L \cup R \cup Q_j) - r(L \cup R) \\
& = &  5 + 5 - 5 = 5.
\end{eqnarray*}
Thus \ref{sub-2} holds. 

First suppose that $n = 4$. Then, by Lemma~\ref{getting-interesting},  
$$r(L \cup Q_1 \cup Q_2 \cup Q_3 \cup Q_4)  = r(M) = 7.$$

Next we show the following.

\begin{sublemma}
\label{sub-1}
For $\{g,h,i,j\} = \{1,2,3,4\}$, if   $\sqcap(Q_g, Q_h) =  0$, then  $\sqcap^*(Q_i, Q_j) =  1$. 
\end{sublemma}

To see this, observe that $r(Q_g \cup Q_h) =  4$  as $\sqcap(Q_g, Q_h) =  0$. Now 
$r(L \cup Q_g \cup Q_h) =  5 =  r(R \cup Q_g \cup Q_h)$. Thus
 \begin{eqnarray*}
r(L \cup R \cup Q_g \cup Q_h) & \le &   r(L  \cup Q_g \cup Q_h) + r(R \cup Q_g \cup Q_h) - r(Q_g \cup Q_h)\\
& = &  5 + 5 - 4\\
& = & r(M) - 1.
\end{eqnarray*} 
Hence $Q_i \cup Q_j$ contains a cocircuit of $M$. Because each of $Q_i$ and $Q_j$ consists of a clonal pair of elements, if 
$\cl(L \cup R \cup Q_g \cup Q_h)$ meets $Q_i$, then it contains $Q_i$. In that case, $\lm(Q_j) \le 1$, a contradiction. We conclude that 
$Q_i \cup Q_j$ is a cocircuit of $M$. Thus, by Lemma~\ref{clonal_core}(iii), $\sqcap^*(Q_i, Q_j) = 1$. Hence  \ref{sub-1} holds.

\begin{sublemma}
\label{sub00}
If  $\sqcap(L, R \cup Q_h) =  2$, then  $\sqcap^*(Q_i, Q_j) =  1$ for all distinct $i$ and $j$ in $\{1,2,3,4\} - \{h\}$. 
\end{sublemma}

By Lemma~\ref{clonal_core}(iii) again, $\sqcap^*(Q_i, Q_j) =  1$ if and only if $Q_i \cup Q_j$ is a cocircuit of $M$. 
Since $\sqcap(L, R \cup Q_h) =  2$, we have 
 \begin{eqnarray*}
r(L \cup R \cup Q_h) & = &   r(L) + r(R \cup Q_h) - 2\\
& = &  3 + 4 - 2 = 5.
\end{eqnarray*}
Then, for $i$ in $\{1,2,3,4\}- \{h\}$, as $\sqcap(Q_i, L \cup R \cup Q_h) \ge \sqcap(Q_i, L)  = 1$, it follows that 
$r(L \cup R \cup Q_h \cup Q_i) \le 6$. Thus $Q_j \cup Q_k$ contains a cocircuit where $\{h,i,j,k\} = \{1,2,3,4\}$. 

Continuing with the proof of \ref{sub00}, we now show that 
\begin{sublemma}
\label{sub11}
$r(L \cup R \cup Q_h \cup Q_i) = 6$ for all $i$ in $\{1,2,3,4\} - \{h\}$. 
\end{sublemma}

Assume that $r(L \cup R \cup Q_h \cup Q_i) = 5$ for some $i$ in $\{1,2,3,4\} - \{h\}$. Then $r(L \cup R \cup Q_h \cup Q_i \cup Q_j)  \le 6$ for $\{j,k\} = \{1,2,3,4\} - \{h,i\}$, so 
$\lm(Q_k) \le 1$, a contradiction. Thus \ref{sub11} holds.

It follows as above that, for  $\{j,k\} = \{1,2,3,4\} - \{h,i\}$, since each of $Q_j$ and $Q_k$ consists of a clonal pair of elements, $L \cup R \cup Q_h \cup Q_i$ is a hyperplane, so 
$Q_j \cup Q_k$ is a cocircuit. Hence $\sqcap^*(Q_j,Q_k) = 1$. Thus \ref{sub00} holds.

Next we show the following.
\begin{sublemma}
\label{sub22}
If  $\sqcap(L, R \cup Q_h) =  2$, then  $\sqcap(L, R \cup Q_i) =  1$ for all distinct $i$  in $\{1,2,3,4\} - \{h\}$. 
\end{sublemma}

Assume 
$\sqcap(L, R \cup Q_i) =  2$. Then,  by \ref{sub-2},  $r(L \cup R \cup Q_h \cup Q_i) \le 5$. 
But this contradicts \ref{sub11}. Thus \ref{sub22} holds.

Now 
$\sqcap(L,R \cup Q_i) \ge \sqcap(L,R) = 1$. Moreover, 
\begin{equation}
\label{aacc}
\sqcap(L,R \cup Q_i) + \sqcap^*(L,R \cup Q_i) \le 3
\end{equation}
since 
\begin{eqnarray*}
\sqcap(L,R \cup Q_i) + \sqcap^*(L,R \cup Q_i) & = &   \lm(L) + \lm(R \cup Q_i) - \lm(L \cup R \cup Q_i)\\
& \le &  3 + 3 - 3 = 3.
\end{eqnarray*}

By \ref{sub22}, duality, and (\ref{aacc}), we may assume the following without loss of generality. 

\begin{sublemma}
\label{sub33}
If  $\sqcap^*(L, R \cup Q_i) \neq  1$, then  $i = 1$ and $\sqcap^*(L, R \cup Q_i) = 2$.
If  $\sqcap(L, R \cup Q_j) \neq  1$, then  $j = 2$ and $\sqcap(L, R \cup Q_j) = 2$. 
\end{sublemma}

%It follows that 
%\begin{sublemma}
%\label{sub44}
%$\sqcap(L, R \cup Q_k) = 1 = \sqcap^*(L, R \cup Q_k) $ for each $k$ in $\{3,4\}$.
%\end{sublemma}

Next we show the following.
\begin{sublemma}
\label{sub55}
If $\sqcap^*(L, R \cup Q_1) = 2$, then 
$\sqcap(Q_i, Q_j) = 1$ and  $\sqcap^*(Q_i, Q_j) = 0$ for all distinct $i$ and $j$ in $\{2,3,4\}$.
\end{sublemma}

We have $\sqcap(L,R \cup Q_1) \ge \sqcap(L,R) = 1.$ 
By (\ref{aacc}), $\sqcap(L,R \cup Q_1) \le 3 - \sqcap^*(L,R \cup Q_1) = 1$. 
Thus $\sqcap(L,R \cup Q_1) = 1$. Hence, for the $(4,2)$-flexipath $(L,Q_2,Q_3,Q_4, R \cup Q_1)$, we have  
$\sqcap^*(L, R \cup Q_1) = 2$ and $\sqcap(L, R \cup Q_1) = 1$. Thus, it follows by Lemma~\ref{1,1,3}(ii) that 
$\sqcap(Q_i, Q_j)=1$ and $\sqcap^*(Q_i,Q_j) = 0$ for all distinct $i$ and $j$ in $\{2,3,4\}$, that is, \ref{sub55} holds.

By \ref{sub55} and duality, we immediately obtain the following. 
\begin{sublemma}
\label{sub66}
If $\sqcap(L, R \cup Q_2) = 2$, then 
$\sqcap^*(Q_i, Q_j) = 1$ and  $\sqcap(Q_i, Q_j) = 0$ for all distinct $i$ and $j$ in $\{1,3,4\}$.
\end{sublemma}

By \ref{sub66} and \ref{sub55} and considering $\sqcap(Q_3,Q_4)$, we deduce that 
$\sqcap(L, R \cup Q_2) = 1$ or $\sqcap^*(L, R \cup Q_1) = 1$. Thus, by \ref{sub33}  
and duality, we may assume that 

\begin{sublemma}
\label{sub77}
$\sqcap(L, R \cup Q_i) = 1$ for all $i$ in $\{1,2,3,4\}$, and $\sqcap^*(L, R \cup Q_j) = 1$ for all $j$ in $\{2,3,4\}$.  Moreover,  
$\sqcap^*(L, R \cup Q_1) \in  \{1,2\}$.
\end{sublemma}

For each $j$ in $\{2,3,4\}$, consider the path $(L,Q_g,Q_h,Q_i, Q_j \cup R)$, which we relabel as $(L,Q_g, Q_h,Q_i, R')$. 
Then 
$\sqcap(L,R') = 1 = \sqcap^*(L,R').$ We take the clonal core of $(M, (L,Q_g, Q_h,Q_i, R'))$. It is a rank-6 matroid $M'$. 
By Lemma~\ref{1,1,3}, for each $j$ in $\{2,3,4\}$, there are distinct elements $s$ and $t$ in   $\{1,2,3,4\} - \{j\}$ such that 
$\sqcap(Q_s, Q_t) = 0 = \sqcap^*(Q_s, Q_t)$. Then, by \ref{sub-1}, $\sqcap(Q_p, Q_q) = 1 = \sqcap^*(Q_p, Q_q)$ 
where $\{p,q\} = \{1,2,3,4\} - \{s,t\}$. This contradicts Lemma~\ref{C}. 
%By Lemma~\ref{1,1,3}, one of $\sqcap(L,R \cup Q_t)$ or $\sqcap^*(L,R \cup Q_t)$ is $2$, 
%combining \ref{sub22} and 
%\ref{sub77}. 
We conclude that $\QQ$ does not have exactly four internal steps.

We now consider $(L,Q_1,Q_2,\dots,Q_n,R)$ where $n \ge 4$ and $\sqcap(L,R) = 1 = \sqcap^*(L,R)$. We prove by induction on $n$ that such a path of $4$-separations does not exist. 
We proved this above for $n = 4$. Assume it is true when the path has fewer than $n$ internal steps and suppose that it has exactly $n$ internal steps where $n \ge 5$. 
We continue to operate in the clonal core and to label this clonal core as $M$. As $\sqcap (L,Q_i) = 1$ for all $i$, it follows that 
\begin{equation}
\label{newboy0}
r(M) = n+3.
\end{equation}

\begin{sublemma}
\label{sub88}
If $\sqcap(L, R \cup Q_1) = 2$, then  $\sqcap(L, R \cup Q_i) = 1$ for all $i$ in $\{2,3,\dots,n\}$.  
\end{sublemma}

Assume that $\sqcap(L, R \cup Q_2) = 2$. Then, by \ref{sub-2}, 
\begin{equation}
\label{newboy1}
r(L \cup R \cup Q_1 \cup Q_2) \le 5.
\end{equation}
As $\sqcap(L,R \cup Q_1) = 2$, by considering the path 
 $(L,Q_2,Q_3,\dots,Q_n,R  \cup Q_1)$ of $4$-separations, which has at least four internal steps, we deduce by Lemma~\ref{bign} that 
 \begin{equation}
\label{newboy2}
r(Q_3 \cup Q_4 \cup \dots \cup Q_n) = n.
\end{equation}
Then, by (\ref{newboy1}), (\ref{newboy2}), and (\ref{newboy0}), 
\begin{eqnarray*}
3& \le  &   \lm(Q_3 \cup Q_4 \cup \dots \cup Q_n)\\
& \le & n + 5 - (n+3) = 2, 
\end{eqnarray*}
a contradiction. 
Thus \ref{sub88} holds. 

By \ref{sub88}, duality, and symmetry, we may assume that 
$\sqcap(L,R \cup Q_n) = 1 = \sqcap^*(L,R \cup Q_n).$ 
Then the path $(L,Q_1,Q_2, \dots,Q_{n-1},R \cup Q_n)$ is a path of $4$-separations that violates the induction assumption. 
The lemma now follows by induction.
\end{proof}

\begin{lemma}
\label{the-plot-thickens}
Let $\QQ$  be a $(4,2)$-flexipath $(L,Q_1,Q_2,\ldots,Q_n,R)$ in a matroid $M$, where $n \ge 2$ but $n \neq 3$.
Assume that $\QQ$ is neither squashed nor stretched and has no specially placed steps.
Then exactly one of the following holds for all distinct $i$ and $j$ in $[n]$.
\begin{itemize}
\item[(i)] $\sqcap(Q_i,Q_j)=0$ and $\sqcap^*(Q_i,Q_j)=1$.
\item[(ii)] $\sqcap(Q_i,Q_j)=1$ and $\sqcap^*(Q_i,Q_j)=0$.
\item[(iii)] $n= 2$ and $\sqcap(Q_i,Q_j)=0 = \sqcap^*(Q_i,Q_j)$, while $\sqcap(L,R) = 1 = \sqcap^*(L,R)$.
\end{itemize}
\end{lemma}

\begin{proof}
By Lemma~\ref{C}, for a given pair $i,j$, we must either have one of the outcomes described
in the lemma, or
\begin{equation}
\label{qiqj}
\sqcap(Q_i,Q_j)=0 = \sqcap^*(Q_i,Q_j).
\end{equation}
It remains to prove that we have the same outcome for all such pairs and that, when (\ref{qiqj}) arises, $n = 2$. By Lemmas~\ref{LR0} and \ref{bign}, since $n \neq 3$, 
\begin{itemize}
\item[(a)] $\sqcap(L,R) = 2$ and $\sqcap^*(L,R) = 1$; or
\item[(b)] $\sqcap^*(L,R) = 2$ and $\sqcap(L,R) = 1$; or 
\item[(c)] $n = 2$ and $\sqcap(L,R) = 1 = \sqcap^*(L,R)$.
\end{itemize}

Suppose that (c) holds. Then, since we may view $M$ as a clonal-core matroid, we have 
$r(L \cup R) = 5$, so $r(M) \ge 5$. Now, by Lemma~\ref{getting-interesting}, $r(L\cup Q_1 \cup Q_2) = 5$ and 
$$3 = \lm(R) = r(L \cup Q_1 \cup Q_2) + r(R) - r(M).$$ 
Thus $r(M) = r(L \cup Q_1 \cup Q_2) = 5 = r(L \cup R).$ Hence $Q_1 \cup Q_2$ is coindependent in  $M$. By Lemma~\ref{clonal_core}, we deduce that $\sqcap^*(Q_1,Q_2) = 0$. By duality, $\sqcap(Q_1,Q_2) = 0$. 

By duality, we may now assume that (a) holds. We  may also assume that $M$ is a clonal-core matroid.
Then, by Lemma~\ref{bign}(v), $\sqcap(Q_i,Q_j)=0$ for all distinct $i$ and $j$ in $[n]$. Now, fix $i$ and $j$, and let $J = [n] - \{i,j\}$.  
Then, by Lemma~\ref{bign}(vi) and (iv), 
\begin{eqnarray*}
r(M) - r(L \cup R \cup Q_J) & = &    r(L) + \sum_{h = 1} ^ n r(Q_h) + r(R) - n - 3\\
&& - (r(L) + \sum_{h \in J}  r(Q_h) + r(R) - (n- 2) - 2)\\
& = & r(Q_i) + r(Q_j) - 3\\
& = & 2 + 2 - 3 = 1.
\end{eqnarray*}
We deduce that $Q_i \cup Q_j$ contains a cocircuit of $M$. As each of $Q_i$ and $Q_j$ consists of a pair of clones, $Q_i \cup Q_j$ is a cocircuit of $M$.
Then, by Lemma~\ref{clonal_core}(iii), $\sqcap^*(Q_i,Q_j)=1$. We conclude that, when (a) holds, so does (i). By duality, when (b) holds, so does (ii). Thus, (\ref{qiqj}) never arises. 
\end{proof}

\begin{theorem}
\label{flexi-types}
Let $\QQ$  be a $(4,2)$-flexipath $(L,Q_1,Q_2,\ldots,Q_n,R)$ in a matroid $M$, where $n\geq 2$.
Then  the following hold.
\begin{itemize}
\item[(i)] If $\QQ$ has no specially placed steps,  then either 
\begin{itemize}
\item[(a)] $\QQ$ is squashed, stretched, 
paddle-reminiscent, or spike-reminiscent; or
\item[(b)] $n = 3$ and, in either $M$ or $M^*$,  the $(4,2)$-flexipath $\QQ$ is prism-like, tightened-prism-like, doubly-tightened-prism-like, relaxed-spike-reminiscent,  V\'{a}mos-inspired, or nasty; or 
\item[(c)] $n= 2$ and $\sqcap(Q_i,Q_j)=0 = \sqcap^*(Q_i,Q_j)=0$, while $\sqcap(L,R) = 1 = \sqcap^*(L,R)$.
\end{itemize}
\item[(ii)] If $Q_n$ is a specially placed step of type {\em (S1)}, and $n\geq 3$, then 
$(L,Q_1,\ldots,Q_{n-1},Q_n\cup R)$ is paddle-reminiscent or relaxed-paddle-reminiscent.
\item[(iii)] If $Q_n$ is a specially placed step of type {\em (S2)}, and $n\geq 3$, then 
$(L,Q_1,\ldots,Q_{n-1},Q_n\cup R)$ is spike-reminiscent or relaxed-spike-reminiscent.
\end{itemize}
\end{theorem}

\begin{proof} Suppose that $\QQ$ has no specially placed steps and that $\QQ$ is not squashed or stretched. Then, by Lemma~\ref{F}, for all $i$ in $[n]$, 
\begin{equation}
\label{LQ_i}
\sqcap(L,Q_i) = \sqcap^*(L,Q_i) = 1 = \sqcap(R,Q_i) = \sqcap^*(R,Q_i).
\end{equation} 

Suppose that $n \neq 3$. Then, by Lemma~\ref{LR0} and its dual, $\sqcap(L,R) \neq 0$ and $\sqcap^*(L,R) \neq 0$. Thus $\sqcap(L,R) \geq 1$ and $\sqcap^*(L,R) \geq 1$. 
By Lemma~\ref{GG}, 
$$\sqcap(L,R) + \sqcap^*(L,R) \le 3.$$ 
If $\sqcap(L,R) = 1 = \sqcap^*(L,R)$, then, as $n \neq 3$, by Lemma~\ref{1.14}, we get  that 
$n = 2$.  We deduce that either 
\begin{itemize}
\item[(a)] $\sqcap(L,R) = 2$ and $\sqcap^*(L,R) = 1$; or
\item[(b)] $\sqcap^*(L,R) = 2$ and $\sqcap(L,R) = 1$; or 
\item[(c)] $n = 2$ and $\sqcap(L,R) = 1 = \sqcap^*(L,R)$.
\end{itemize}

Suppose that (c) holds. Then, by Lemma~\ref{the-plot-thickens}, $\sqcap(Q_1,Q_2)=0 = \sqcap^*(Q_1,Q_2)$.
 If (a) holds, then, by Lemma~\ref{bign}(v), 
  $\sqcap(Q_i,Q_j)=0$ for all distinct $i$ and $j$. Thus, by Lemma~\ref{the-plot-thickens}, $\sqcap^*(Q_i,Q_j)=1$ for all distinct $i$ and $j$. We deduce that 
$\QQ$ is paddle-reminiscent. By duality, if (b) holds, then $\QQ$ is spike-reminiscent. 

Now let $n = 3$ and assume that $\QQ$ is neither paddle-reminiscent nor spike-reminiscent.  By Lemma~\ref{GG}, $\sqcap(L,R) + \sqcap^*(L,R) \le 3$. By duality, we may assume that $\sqcap(L,R) \le \sqcap^*(L,R)$. If $\sqcap(L,R) = 0$, then the possibilities for $\QQ$ are identified in Lemma~\ref{LR0}, namely, $\QQ$ is prism-like, relaxed-spike-reminiscent,  tightened-prism-like,  or V\'{a}mos-inspired. 
We may now assume that $\sqcap(L,R) = 1$.  Then, by Lemma~\ref{1,1,3cor}, $\sqcap^*(L,R) = 1$ and the possibilities for $\QQ$ are identified in (i) of that lemma. In particular, $\QQ$ is doubly-tightened-prism-like or is nasty. 

By duality, it only remains to prove (ii). Assume $Q_n$  is a specially placed step of type (S1) and that $n \ge 3$. Then $(L,Q_1,Q_2,\dots, Q_{n-1}, Q_n \cup R\}$ is a $(4,2)$-flexipath $\QQ'$. Suppose $\QQ'$ has a specially placed element $Q_i$. Assume first that $Q_i$ is of type (S1). Then  $\sqcap(L,Q_i) = 2$, so, by Lemma~\ref{A}, $Q_i$ is specially placed in $\QQ$. Thus $\QQ$ has two specially placed elements, a contradiction to Lemma~\ref{notmany}. Thus $Q_i$ is specially placed of type (S2). Then 
$\sqcap^*(L,Q_i) = 2$, so, again, $Q_i$ is specially placed in $\QQ$, a contradiction. We conclude that $\QQ'$ has no specially placed steps.

We now argue in the clonal core. Because $Q_n$  is a specially placed step of type (S1), $\sqcap(L,R) = 2$, so $r(L \cup R) = 4$. Also $\sqcap(R,Q_n) = 2$, so $r(Q_n \cup R) = 3$ and $r(L \cup Q_n \cup R) = 4$. Thus
$$\sqcap(L,Q_n \cup R) = r(L) + r(Q_n \cup R) - r(L \cup Q_n \cup R) = 3 + 3 -4 =2.$$
Hence $\QQ'$  is neither squashed nor stretched. By Lemmas~\ref{A} and \ref{F},  $\sqcap(L,Q_i) = \sqcap(R,Q_i) = 1 = \sqcap^*(L,Q_i) = \sqcap^*(R,Q_i)$ for all $i$ in $[n-1]$.  If $\sqcap^*(L,Q_n \cup R)  = 0$, then, by the dual of Lemma~\ref{LR0}(ii),   $n-1 = 3$, so $n= 4$. Moreover, as $\sqcap(L,Q_n \cup R) = 2$, it follows by the dual of Lemma~\ref{LR0}(v) and (viii), that $\QQ'$ is relaxed-paddle-reminiscent. 
If $\sqcap^*(L,Q_n \cup R)  = 1$, then, by the dual of Lemma~\ref{1,1,3}(ii), $\QQ'$ is paddle-reminiscent. Thus (ii) holds.
\end{proof}

The complexity of the last result can be simplified by classifying the numerous outcomes by a more succinct list of defining characteristics.

\begin{corollary}
\label{char}
Let $\QQ$  be a $(4,2)$-flexipath $(L,Q_1,Q_2,\ldots,Q_n,R)$ in a matroid, where $n\geq 2$ and $\QQ$ has no specially placed steps. 
For $\sqcap(L,R) \le \sqcap^*(L,R)$,  the following outcomes are possible.
\begin{itemize}
\item[(i)] If $(\sqcap(L,R), \sqcap^*(L,R)) = (0,0)$, then $\QQ$ is prism-like.
\item[(ii)] If $(\sqcap(L,R), \sqcap^*(L,R)) = (0,1)$, then $n = 3$ and 
\begin{itemize} 
\item[(a)] $\sqcap(Q_i, Q_j) = 0$ for all distinct $i$ and $j$, and $\QQ$ is tightened-prism-like; or 
\item[(b)] $\sqcap(Q_i, Q_j) = 1$ for exactly one distinct pair $\{i,j\}$, and $\QQ$ is V\'{a}mos-inspired.
\end{itemize}
\item[(iii)] If $(\sqcap(L,R), \sqcap^*(L,R)) = (0,2)$, then $\QQ$ is relaxed-spike-reminiscent.
\item[(iv)] If $(\sqcap(L,R), \sqcap^*(L,R)) = (0,3)$, then $\QQ$ is stretched.
\item[(v)] If $(\sqcap(L,R), \sqcap^*(L,R)) = (1,1)$, then $n \in \{2, 3\}$.
\item[(vi)] If $(\sqcap(L,R), \sqcap^*(L,R)) = (1,1)$ and $n = 2$, then 
$\sqcap(Q_1, Q_2) = 0 = \sqcap^*(Q_1, Q_2)$. 
\item[(vii)] If $(\sqcap(L,R), \sqcap^*(L,R)) = (1,1)$ and $n = 3$, then 
\begin{itemize} 
\item[(a)]  $\sqcap(Q_i, Q_j) = 0 = \sqcap^*(Q_i, Q_j)$ for all distinct $i$ and $j$, and $\QQ$ is doubly-tightened-prism-like; or 
\item[(b)] the multiset of pairs $\{(\sqcap(Q_i, Q_j),\sqcap^*(Q_i, Q_j)); i \neq j\}$ contains 
\begin{itemize}
\item[(1)] both $(0,1)$ and $(1,0)$ and $\QQ$ is mixed nasty; or 
\item[(2)] $(1,0)$ but not  $(0,1)$ and $\QQ$ is plane nasty; or 
\item[(3)] $(0,1)$ but not  $(1,0)$ and $\QQ$ is dual-plane nasty.
\end{itemize}
\end{itemize}
\item[(viii)] If $(\sqcap(L,R), \sqcap^*(L,R)) = (1,2)$, then $\QQ$ is spike-reminiscent.
\end{itemize}
\end{corollary}

To see an example satisfying (vi), we can modify a prism-like matroid as follows. Take a 6-element independent set $\{b_1,b_2,\dots, b_6\}$. Add $b_1',b_2'$, and $b_3'$ freely on the flat spanned by $\{b_1,b_2,b_3\}$ and add $b_4',b_5'$, and $b_6'$  freely on the flat spanned by $\{b_4,b_5,b_6\}$. Add a point $c$ freely on the line spanned by $\{b_3,b_6\}$. Add points $c_1$ and $c_4$ freely on the line spanned by $\{b_1,b_4\}$. Add  points $c_2$ and $c_5$ freely on the line spanned by $\{b_2,b_5\}$. Contract $c$ and delete $\{b_1,b_2,\dots, b_6\}$ to get a rank-5 matroid $M$.
Let $(L,R)  = (\{b_1',b_2',b_3'\}, \{b_4',b_5',b_6'\})$ and 
$(Q_1,Q_2) = \{c_1,c_4\}, \{c_2,c_5\})$.  Then 
$\sqcap(L,R)  = 1$ so  
$r(L \cup R) = 5 =  r(M)$. Also $Q_1 \cup Q_2$ is neither a circuit nor a cocircuit so 
$\sqcap(Q_1,Q_2)  = 0 = \sqcap^*(Q_1,Q_2)$. 
Finally, $r^*(L\cup R) = |L\cup R| + r(Q_1 \cup Q_2) - r(M) = 6 + 4 -5 = 5$. It follows that $\sqcap^*(L,R)  = 1$. 

We conclude by noting that Theorem~\ref{mainlite} follows from Theorem~\ref{flexi-types}. 

\begin{proof}[Proof of Theorem~\ref{mainlite}.] 
By Lemma~\ref{notmany}, when we absorb any specially placed steps of $\QQ$ into its right end, we get a $(4,2)$-flexipath $\QQ'$ with at least four internal steps none of which is specially placed. The theorem now follows immediately from Theorem~\ref{flexi-types}(i). 
\end{proof}

\section*{Acknowledgements}

The authors thank the  anonymous referees who helped us to both sharpen the focus of the paper and to correct 
 a number of errors.

\end{document}